\documentclass[12pt]{amsart}

\usepackage{amsfonts,amsthm,amsmath,amssymb,amscd,mathrsfs,tikz}
\usepackage{graphics}
\usepackage{indentfirst}
\usepackage{cite}
\usepackage{latexsym}
\usepackage[dvips]{epsfig}
\usepackage{color}
\usepackage{enumerate}
\usepackage{tikz}
\usetikzlibrary{patterns}
\newtheorem{theorem}{Theorem}[section]
\newtheorem{remark}{Remark}[section]

\newtheorem{lemma}[theorem]{Lemma}

\newcommand{\bt}{\begin{theorem}}
	\newcommand{\bl}{\begin{lemma}}
		\newcommand{\el}{\end{lemma}}
	\newcommand{\et}{\end{theorem}}

\newcommand{\bn}{\begin{eqnarray}}
	\newcommand{\en}{\end{eqnarray}}
\newcommand{\bnn}{\begin{eqnarray*}}
	\newcommand{\enn}{\end{eqnarray*}}

\newcommand{\ba}{\begin{aligned}}
	\newcommand{\ea}{\end{aligned}}
\newcommand{\be}{\begin{equation}}
	\newcommand{\ee}{\end{equation}}

\newcommand{\bBV}{\boldsymbol{V}}
\newcommand{\Bn}{{\boldsymbol{n}}}
\newcommand{\Bt}{{\boldsymbol{\tau}}}
\newcommand{\Bu}{{\boldsymbol{u}}}
\newcommand{\Be}{{\boldsymbol{e}}}

\newcommand{\D}{{\boldsymbol{D}}}

\newcommand{\Ps}{\mathbf{\Psi}}

\newcommand{\Bx}{{\boldsymbol{x}}}

\newcommand{\OR}{{\mathscr{O}}_{R}}

\newcommand{\wtDR}{\widetilde{D_{R}}}
\newcommand{\wtOR}{\widetilde{\mathscr{O}_{R}}}

\begin{document}
		
		\title[Liouville-type theorems]
		{Liouville-type theorems for Axisymmetric solutions to steady  Navier-Stokes system in a  layer domain}

		\author{Jingwen Han}
		\address{School of Mathematical Sciences, Shanghai Jiao Tong University, 800 Dongchuan Road, Shanghai, China}
		\email{hjw126666@sjtu.edu.cn}

		\author{Yun Wang}
		\address{School of Mathematical Sciences, Center for dynamical systems and differential equations, Soochow University, Suzhou, China}
		\email{ywang3@suda.edu.cn}

		\author{Chunjing Xie}
		\address{School of mathematical Sciences, Institute of Natural Sciences,
			Ministry of Education Key Laboratory of Scientific and Engineering Computing,
			and CMA-Shanghai, Shanghai Jiao Tong University, 800 Dongchuan Road, Shanghai, China}
		\email{cjxie@sjtu.edu.cn}

		\begin{abstract}
			In this paper, we investigate the Liouville-type theorems for axisymmetric solutions to steady Navier-Stokes system in a  layer domain. The both cases for the flows supplemented  with no-slip boundary and Navier boundary  conditions are studied. If the width of the outlet grows at a rate less than $R^{\frac{1}{2}}$, any bounded solution is proved to be trivial. Meanwhile,  if the width of the outlet grows at a rate less than  $R^{\frac{4}{5}}$, every D-solution is proved to be trivial. The key idea of the proof  is to establish  a Saint-Venant  type estimate that characterizes the growth of Dirichlet integral of nontrivial solutions.
		\end{abstract}


		
		\maketitle

		\section{Introduction and Main Results}
		In this paper, we are interested with the Liouville-type theorems for solutions to the three-dimensional steady incompressible Navier-Stokes system,
		\begin{equation}\label{eqsteadyns}
			\left\{ \ba
			&-\Delta \Bu + (\Bu \cdot \nabla )\Bu + \nabla P  = 0, \ \ \ \ \
			&\mbox{in}\ \Omega,\\
			& \nabla \cdot \Bu =0,  \ \ \ \ \ &\mbox{in}\ \Omega,\\
			\ea \right.
		\end{equation}
		where the unknown function $\Bu=(u^1, u^2, u^3)$  is the velocity field,  $P$ is the pressure and  the domain $\Omega$ is assumed to be a three-dimensional  layer (see Figure \ref{Fig1-1}),
		\begin{equation}\label{domaapertu2}
			\Omega=\left\{\Bx=(x_{1}, x_{2}, x_{3})\in \mathbb{R}^{3}:  \Bx_{h}=(x_{1},x_{2})\in\mathbb{R}^{2}, \  f_{1}(\Bx_{h})<x_{3}<f_{2}(\Bx_{h}) \right\},
		\end{equation}
		where  $f_{i}(\Bx_{h})$ ($i=1,2$) is smooth. Denote $d=\inf\limits_{\Bx_{h}\in \mathbb{R}^2}\left(f_{2}(\Bx_{h})-f_{1}(\Bx_{h})\right)>0$. Without loss of generality, we assume that $d=1$. If $f_{1}(\Bx_{h})=0$ and $f_{2}(\Bx_{h})=1$, the domain $\Omega=\mathbb{R}^2 \times (0, 1)$ is called a three-dimensional slab.
		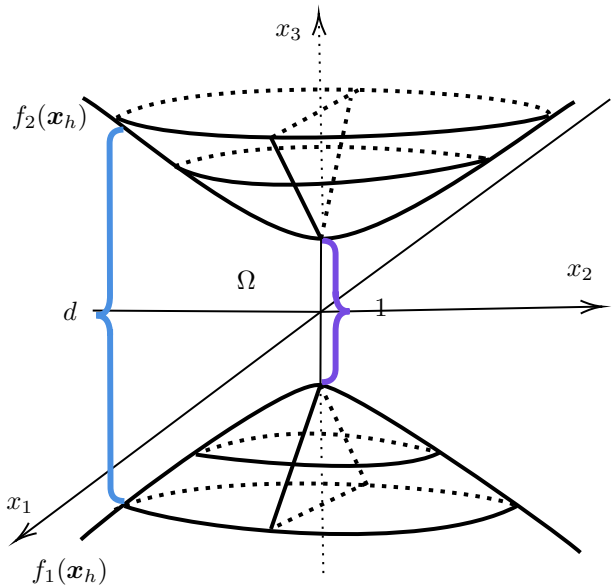
\begin{figure}[h]
			\begin{center}
				\centering

				\tikzset{every picture/.style={line width=0.75pt}} 
				
				\begin{tikzpicture}[x=0.75pt,y=0.75pt,yscale=-1,xscale=1]
					
					\draw    (309.33,250.33) -- (449.33,247.7) ;
					\draw [shift={(451.33,247.67)}, rotate = 178.92] [color={rgb, 255:red, 0; green, 0; blue, 0 }  ][line width=0.75]    (10.93,-3.29) .. controls (6.95,-1.4) and (3.31,-0.3) .. (0,0) .. controls (3.31,0.3) and (6.95,1.4) .. (10.93,3.29)   ;
					\draw    (309.33,250.33) -- (156.93,364.8) ;
					\draw [shift={(155.33,366)}, rotate = 323.09] [color={rgb, 255:red, 0; green, 0; blue, 0 }  ][line width=0.75]    (10.93,-3.29) .. controls (6.95,-1.4) and (3.31,-0.3) .. (0,0) .. controls (3.31,0.3) and (6.95,1.4) .. (10.93,3.29)   ;
					\draw  [dash pattern={on 0.84pt off 2.51pt}]  (310.33,201) -- (308.37,101.67) ;
					\draw [shift={(308.33,99.67)}, rotate = 88.87] [color={rgb, 255:red, 0; green, 0; blue, 0 }  ][line width=0.75]    (10.93,-3.29) .. controls (6.95,-1.4) and (3.31,-0.3) .. (0,0) .. controls (3.31,0.3) and (6.95,1.4) .. (10.93,3.29)   ;
					\draw  [dash pattern={on 0.84pt off 2.51pt}]  (309.17,287.27) -- (310.67,382.72) ;
					\draw    (455.33,143) -- (309.33,250.33) ;
					\draw    (309.33,250.33) -- (194.33,249.67) ;
					\draw [line width=1.5]    (189.33,142.33) .. controls (207.33,152.33) and (273.5,212.73) .. (309.5,213.4) .. controls (345.5,214.07) and (422.17,152.93) .. (437.33,144.33) ;
					\draw [line width=1.5]    (188.33,365.33) .. controls (208.33,347.33) and (291.17,285.27) .. (309.17,287.27) .. controls (327.17,289.27) and (428.33,360.83) .. (440.33,371.83) ;
					\draw  [color={rgb, 255:red, 118; green, 74; blue, 226 }  ,draw opacity=0.95 ][line width=2.25]  (309.83,285.83) .. controls (314.5,285.83) and (316.83,283.5) .. (316.83,278.83) -- (316.83,260.08) .. controls (316.83,253.41) and (319.16,250.08) .. (323.83,250.08) .. controls (319.16,250.08) and (316.83,246.75) .. (316.83,240.08)(316.83,243.08) -- (316.83,221.33) .. controls (316.83,216.66) and (314.5,214.33) .. (309.83,214.33) ;
					\draw [color={rgb, 255:red, 0; green, 0; blue, 0 }  ,draw opacity=1 ][line width=0.75]    (309.5,213.4) -- (309.17,287.27) ;
					\draw [line width=1.5]  [dash pattern={on 1.69pt off 2.76pt}]  (309.5,213.4) -- (324.5,142.4) ;
					\draw [line width=1.5]    (284.28,162.51) -- (309.5,213.4) ;
					\draw [line width=1.5]  [dash pattern={on 1.69pt off 2.76pt}]  (328.33,138.33) -- (284.28,162.51) ;
					\draw [line width=1.5]    (309.17,287.27) -- (284.23,359.93) ;
					\draw [line width=1.5]  [dash pattern={on 1.69pt off 2.76pt}]  (309.17,287.27) -- (334.17,340.93) ;
					\draw [line width=1.5]  [dash pattern={on 1.69pt off 2.76pt}]  (332.33,336.83) -- (284.23,359.93) ;
					\draw [line width=1.5]    (206.33,152.33) .. controls (241.07,169.13) and (390.33,162.13) .. (424.33,151.33) ;
					\draw [line width=1.5]    (235.83,176.83) .. controls (269.83,194.83) and (357.13,182.93) .. (394.33,173.33) ;
					\draw [line width=1.5]    (246.33,322.33) .. controls (275.13,326.73) and (347.2,333.23) .. (369.33,321.33) ;
					\draw [line width=1.5]  [dash pattern={on 1.69pt off 2.76pt}]  (235.83,176.83) .. controls (295.83,154.83) and (423.33,177.67) .. (388.33,172.33) ;
					\draw [line width=1.5]  [dash pattern={on 1.69pt off 2.76pt}]  (206.33,152.33) .. controls (247.33,127.79) and (432.8,140.07) .. (424.33,151.33) ;
					\draw [line width=1.5]  [dash pattern={on 1.69pt off 2.76pt}]  (251.33,320.33) .. controls (295,305.17) and (334.8,313.07) .. (369.33,321.33) ;
					\draw [line width=1.5]  [dash pattern={on 1.69pt off 2.76pt}]  (205.33,350.33) .. controls (252.83,327.17) and (389.83,339.17) .. (408.33,348.33) ;
					\draw  [color={rgb, 255:red, 74; green, 144; blue, 226 }  ,draw opacity=1 ][line width=2.25]  (209.83,158.33) .. controls (205.16,158.34) and (202.84,160.68) .. (202.85,165.35) -- (203.05,242.1) .. controls (203.07,248.77) and (200.75,252.11) .. (196.08,252.12) .. controls (200.75,252.11) and (203.09,255.43) .. (203.1,262.1)(203.1,259.1) -- (203.31,338.85) .. controls (203.32,343.52) and (205.65,345.84) .. (210.32,345.83) ;
					\draw [line width=1.5]    (210.33,348.33) .. controls (247.67,361.33) and (362.33,372.33) .. (408.33,348.33) ;
					
					\draw (149.33,342.73) node [anchor=north west][inner sep=0.75pt]  [font=\footnotesize]  {$x_{1}$};
					\draw (432,225.07) node [anchor=north west][inner sep=0.75pt]  [font=\footnotesize]  {$x_{2}$};
					\draw (285,103.73) node [anchor=north west][inner sep=0.75pt]  [font=\footnotesize]  {$x_{3}$};
					\draw (152,143.9) node [anchor=north west][inner sep=0.75pt]  [font=\footnotesize]  {$f_{2}(\boldsymbol{x}_{h})$};
					\draw (162,373.4) node [anchor=north west][inner sep=0.75pt]  [font=\footnotesize]  {$f_{1}(\boldsymbol{x}_{h})$};
					\draw (335,242.4) node [anchor=north west][inner sep=0.75pt]  [font=\footnotesize]  {$1$};
					\draw (265.5,228.23) node [anchor=north west][inner sep=0.75pt]  [font=\footnotesize]  {$\Omega $};
					\draw (178,243.73) node [anchor=north west][inner sep=0.75pt]  [font=\footnotesize]  {$d$};

				\end{tikzpicture}

			\end{center}
			\caption{The 3D layer $\Omega$}
			\label{Fig1-1}

		\end{figure}

		The analysis for steady Navier-Stokes system was pioneered   by Leray in \cite{LEJMPA33}.
		If a weak solution $\Bu$ to \eqref{eqsteadyns} satisfies
		\begin{equation}\label{Dcon1}
			\int_{\Omega}|\nabla\Bu|^{2}
			\,d\Bx<+\infty,
		\end{equation}
		it is called a D-solution of steady Navier-Stokes system. Leray proved  the existence of D-solutions in both bounded and exterior domains. However, the uniqueness  is still open. As  pointed out in \cite[X.9 Remark X.9.4]{GAGP11}, a closely related and longstanding challenging problem is whether the D-solution of the homogeneous Navier-Stokes system in $\mathbb{R}^{3}$ is trivial if it vanishes at infinity.  This Liouville problem for D-solutions of three-dimensional Navier-Stokes system in the whole space has attracted a lot of attention in the last a few decades, one may refer to \cite{CWDCDS16,KTWJFA17,SN16,WJDE19,NZ19} and references therein for the progress. There has been some remarkable progress for the two-dimensional flow and the  axisymmetric flow without swirl, owing to the special structure of the equation for vorticity. It was proved in \cite{GWASP78}  that every D-solution in $\mathbb{R}^2$ is trivial. This kind of Liouville-type theorem for axisymmetric D-solutions without swirl in $\mathbb{R}^3$ was established in \cite{KPRJMFM15}. 
		
		Besides the whole space domain, the flows in domains with a noncompact boundary also 
		have both many physical applications and mathematical challenges.
		The existence  of weak solutions of steady Stokes system and Navier-Stokes system was obtained in  \cite{PileJSM84,SANPJMFM991} in the  layer-like domain (i.e., $\Omega$ is of form \eqref{domaapertu2} and satisfies $\Omega\cap\left\{|\Bx|>R_{0}\right\}
		=\left(\mathbb{R}^{2}\times(0,1)\right)\cap\left\{|\Bx|>R_{0}\right\}$ for some $R_{0}$ sufficiently large.), supplemented with no-slip boundary conditions. Moreover, the uniqueness and asymptotic behavior to Navier-Stokes system were obtained when the data satisfy some smallness assumptions. Later on, it was proved that at far fields, the solutions of steady Navier-Stokes system in a layer-like domain behave  like that of linear Stokes system  (cf. \cite{PMSLI2002,PMSAA10}). 
		In \cite{PZAAA07},  the energy estimates for solutions to the Stokes system were established in a growing layer where $\lim\limits_{|\Bx_{h}|\rightarrow \infty}(f_{2}(\Bx_{h})-f_{1}(\Bx_{h}))=\infty$ for $f_{i}(\Bx_{h})$ in \eqref{domaapertu2}. For more results related to steady Stokes system and Navier-Stokes system in the layer-like domain, one  can refer to \cite{ASHIJMSJ03,NazSM90,SANPJMFM992NE,PMZAAS08} and references therein.
		
		Recently, some Liouville-type results were derived for D-solutions in the slab  and aperture domain. In a slab $\mathbb{R}^{2}\times (0,1)$, it was proved  in \cite{CPZZARMA20}  that the   D-solutions are trivial. Very recently,    this result was improved in \cite{aBGWX} by allowing certain growth of local Dirichlet integral.
		In the addendum of \cite{ZOP22RE},  the Liouville theorems for axisymmetric  D-solutions also hold in an aperture, which is in fact a growing layer domain,  with no-slip boundary conditions. For more references about the Liouville-type theorem for D-solutions in a slab, one may refer to \cite{CPZJFA20,HWXAHE,PJMP21} and the references therein.
		
		Liouville-type theorems   for bounded solutions also receive a lot of attention since they have many applications, such as the analysis for the possible singularity of solutions and the asymptotic behavior of solutions at far fields.  It was proved in \cite{KNSS09}  that every bounded two-dimensional and three-dimensional axisymmetric flows without swirl  in the whole space are constants. The further important progress in \cite{LRZMA22} showed that the bounded axisymmetric solution $\Bu$ must be constant vectors provided that $\Bu$ is periodic in one direction and $ru^{\theta}$ is bounded. Most recently, this Liouville-type theorem was established for bounded steady axisymmetric solutions which are periodic in one direction, and bounded steady helically symmetric solutions (cf. \cite{aBGWX,HWXAHE}). In a slab $\mathbb{R}^2 \times (0, 1)$, the bounded axisymmetric solutions with no-slip boundary conditions or Navier boundary conditions were also proved to be trivial (cf. \cite{aBGWX, HWXAHE}). 
		For more Liouville-type theorems on bounded solutions  of Navier-Stokes system in various domains, please refer to \cite{KTWarXi23, JMFM13bfz,LZZSSM17,PLNonwor20}. 
		
		In this paper, we focus on  Liouville-type theorems for axisymmetric solutions to steady  Navier-Stokes system in an axisymmetric  layer domain. The both cases for flows supplemented with  no-slip boundary conditions and Navier boundary conditions are studied.
		
		Let the standard cylindrical coordinates $(r,\theta,z)$ be defined as follows
		\[
		\Bx=(x_{1}, x_{2}, x_{3})
		=(r\cos\theta, r\sin\theta,z).
		\]
		The velocity $\Bu$ is called axisymmetric if
		\[
		\Bu=u^{r}(r,z)\Be_{r}+u^{\theta}(r,z)
		\Be_{\theta}+u^{z}(r,z)\Be_{z},
		\]
		where $u^{r}, u^{\theta}, u^{z}$ are called radial, swirl and axial velocity, respectively, with
		\[
		\Be_{r}=(\cos\theta,\sin\theta,0), \ \ \ \Be_{\theta}=(-\sin\theta,\cos\theta,0)
		\ \ \ \text{and} \ \ \
		\Be_{z}=(0,0,1).
		\]
		For the axisymmetric flow, the  Navier-Stokes system \eqref{eqsteadyns} becomes
		\begin{equation}\label{axieqsteadyns}
			\begin{cases}
				\left(u^{r}\partial_{r}+u^{z}\partial_{z}\right)u^{r}-\dfrac{(u^{\theta})^{2}}{r}+\partial_{r}P
				=\left(\partial^{2}_{r}+\dfrac{1}{r}\partial_{r}+\partial^{2}_{z}-\dfrac{1}{r^{2}}\right)u^{r},\\[8pt]
				\left(u^{r}\partial_{r}+u^{z}\partial_{z}\right)u^{\theta}+\dfrac{u^{\theta}u^{r}}{r}
				=\left(\partial^{2}_{r}+\dfrac{1}{r}\partial_{r}+\partial^{2}_{z}-\dfrac{1}{r^{2}}\right)u^{\theta},\\[8pt]
				\left(u^{r}\partial_{r}+u^{z}\partial_{z}\right)u^{z}+\partial_{z}P=\left(\partial^{2}_{r}+\dfrac{1}{r}\partial_{r}+\partial^{2}_{z}\right)u^{z},\\[8pt]	
				\partial_{r}u^{r}+\partial_{z}u^{z}+\dfrac{u^{r}}{r}=0.
			\end{cases}
		\end{equation}
		For the axisymmetric layer $\Omega$, for $f_{i}$ $(i=1,2)$ defined in \eqref{domaapertu2}, one has $f_{i}(\Bx_{h})=f_{i}(r)$ with $r=\sqrt{x_{1}^{2}+x_{2}^{2}}$.   Let
		\begin{equation}\label{distfuncf1}
			f(r)=f_{2}(r)-f_{1}(r).
		\end{equation}

		First, we consider the Navier-Stokes system  \eqref{eqsteadyns} in $\Omega$ supplemented  with  no-slip boundary conditions,
		\begin{equation}\label{curboundarynoslip}
			\Bu=0, \quad \text{at} \   x_{3}=f_{i}(r),\  i=1,2.
		\end{equation}
		Our first main result is the Liouville-type theorem for the axisymmetric flow in the axisymmetric layer domain $\Omega$ with no-slip boundary conditions.
		\begin{theorem}\label{th:01}
			Let $\Omega$ be the axisymmetric layer domain  satisfying
			\begin{equation}\label{eqdomainrate}
				\max_{i=1,2}\|f^{\prime}_{i}(r)\|_{C
					(\mathbb{R}_{+})}=\gamma<+\infty
			\end{equation}
			and
			\begin{equation}\label{growboundabe1}
				|f(r)|\leq Cr^{\beta} \quad \text{for some constant \ $\beta\geq0$},
			\end{equation}
			where $f(r)$ is defined in \eqref{distfuncf1}.
			
			Let $\Bu$ be a smooth axisymmetric solution to the  Navier-Stokes system \eqref{eqsteadyns}   supplemented with no-slip boundary conditions \eqref{curboundarynoslip} in  $\Omega$.
			Then $\Bu\equiv 0$ if one of the following conditions holds.
			
			(a) For $\beta\in \left[0,\, \dfrac{1}{2}\right)$ and
			\begin{equation}\label{eqgrowinguveu1}
				\lim\limits_{R\rightarrow \infty}R^{2\beta-1}\sup\limits_{z\in[f_{1}(R),\, f_{2}(R)]}|\Bu(R,z)|=0.
			\end{equation}
			
			(b) For $\beta\in \left[0,\, \dfrac{4}{5}\right)$ and
			\begin{equation}\label{DassumDnoslip1}
				\lim\limits_{R\rightarrow \infty}R^{5\beta-4}	\int_{\Omega_{R}}|\nabla\Bu|^{2}\,d\Bx=0,
			\end{equation}
			where
			\begin{equation}\label{domainRcontrus}
				\Omega_{R} = \left\{
				\Bx \in\mathbb{R}^{3}:  (x_{1},   x_{2})\in B_{R},\,
				f_{1}(r)<x_{3}<f_{2}(r)\right\},
			\end{equation}
			and $r=\sqrt{x_{1}^{2}+x_{2}^{2}}$,
			$B_{R}=\left\{(x_{1}, x_{2}):x_1^2+x_2^2 <R^{2} \right\}$.
		\end{theorem}
		There are a few remarks in order.
		\begin{remark}\label{re1bound}
			When $\beta=0$,   i.e.,
			\[
			f(r)\leq \bar{d}<+\infty \quad
			\text{for  any}\ r\in \mathbb{R},
			\]
			hence the outlet of the layer $\Omega$ is  bounded.
			For this particular case, if either
			\[
			\lim\limits_{R\rightarrow \infty}R^{-1}\sup\limits_{z\in[f_{1}(R),\, f_{2}(R)]}|\Bu(R,z)|=0,
			\]
			or
			\[
			\lim\limits_{R\rightarrow \infty}R^{-4} \int_{\Omega_{R}}|\nabla\Bu|^{2}\,d\Bx=0. 
			\] 
			then  $\Bu$ is trivial. This  can be regarded as a generalization of the Liouville-type theorem for the axisymmetric solution in a slab \cite{aBGWX}.
		\end{remark}
		\begin{remark}
			A particular case of Part (a) of Theorem \ref{th:01} asserts that if	$0\leq \beta<\dfrac{1}{2}$ and $\Bu$ is bounded, then $\Bu$ must be trivial.
		\end{remark}
		
		\begin{remark}\label{Re:13D}
			Assume that $0\leq \beta<\dfrac{4}{5}$. Let $\Bu$ be a  smooth axisymmetric solution to the Navier-Stokes system \eqref{eqsteadyns}  with no-slip boundary conditions \eqref{curboundarynoslip}. If $\Bu$ satisfies
			\[
			\int_{\Omega}|\nabla\Bu|^{2}\,d\Bx
			<+\infty,
			\]
			then $\Bu\equiv 0$. It improves the result obtained in \cite[Theorem 4.16]{ZOP22RE}, where  $\beta$ is required to $0\leq \beta<\dfrac{1}{2}$.
		\end{remark}

		Besides the no-slip boundary conditions,  the following so called Navier boundary conditions was proposed by Navier in \cite{NMARSI27}  to describe the behavior of  viscous flows on the boundary 
		\begin{equation}\label{bounNavboual}
			\Bu\cdot\Bn=0,\ \ \ (\Bn\cdot\D(\Bu)+\alpha\Bu)\cdot \Bt=0, \quad \text{at}\ x_{3}=f_{i}(x_{1}, x_{2}),\ i=1, 2,
		\end{equation}
		where $\D(\Bu)$  is the strain tensor defined by
		$$ \D(\Bu)_{k,j}=
		(\partial_{x_j}u^{k}+\partial_{x_k}u^{j})/2, \quad k, j=1,2,3,
		$$
		and $\alpha \geq  0$ is the friction coefficient which measures the tendency of a fluid over the boundary, $\Bt$ and $\Bn$ are the unit tangent and outer normal vectors on the boundary. 
		If $\alpha=0$,
		the Navier boundary conditions  are also called the full (total) slip boundary conditions. If $\alpha \to \infty$, the Navier boundary conditions formally  reduce to the classical no-slip boundary conditions. 
		
		The non-stationary Navier-Stokes system supplemented with the Navier boundary conditions has been widely investigated in various aspects (cf. \cite{HBV05CpAM,HBV06Cpaa,CQIndin10,DLX18Jmfm,JPKSIAM06,MRLARMA12,XXCPAM07}). For stationary Navier-Stokes system, it was proved in \cite{TACGJDE21,ARJDE14} that the existence and uniqueness of weak and strong solutions  in a bounded domain with Navier boundary conditions in Sobolev spaces. In a pipe with full slip boundary conditions (i.e., $\alpha=0$),  the Liouville-type theorems for the stationary Navier-Stokes system were obtained in \cite{ALPY}. For more references about steady flows  with Navier boundary conditions, one may refer to \cite{ARARMA11,BDCDS10,HBV04ADE,SWX,WXJDE23} and references therein.
		
		Our second result is about the axisymmetric flow in  the axisymmetric layer domain $\Omega$ supplemented with Navier boundary conditions.
		\begin{theorem}\label{th:02}
			Let $\Omega$ be the axisymmetric layer domain  satisfying \eqref{eqdomainrate} and \eqref{growboundabe1}.   Let $\Bu$ be a smooth axisymmetric solution to the  Navier-Stokes system \eqref{eqsteadyns} in   $\Omega$ supplemented with Navier boundary conditions \eqref{bounNavboual}  with $\alpha>0$. If $\|\nabla_{\Bt} \Bn\|_{L^{\infty}(\partial\Omega)}$ is  bounded, 
			then $\Bu\equiv 0$ as long as one of the following two conditions holds.
			
			(a) The condition \eqref{growboundabe1} holds for $\beta\in \left[0,\, \dfrac{1}{2}\right)$ and
			\begin{equation}\label{eqgrowinguveu2}
				\lim\limits_{R\rightarrow \infty}R^{2\beta-1}\sup\limits_{z\in[f_{1}(R),\, f_{2}(R)]}|\Bu(R,z)|=0.
			\end{equation}
			
			(b) The condition \eqref{growboundabe1} holds for $\beta\in \left[0,\, \dfrac{4}{5}\right)$ and
			\begin{equation}\label{DassumDnoslip2new}
				\lim\limits_{R\rightarrow \infty}R^{5\beta-4}\left[\int_{\Omega_{R}}|\nabla\Bu|^{2}\,d\Bx+2\alpha\int_{\partial\Omega_{R}\cap\partial\Omega}|\nabla\Bu|^{2}\,dS\right]=0,
			\end{equation}
			where
			$\Omega_{R}$ is defined in \eqref{domainRcontrus}.
		\end{theorem}
		There are several remarks on Theorem \ref{th:02}.
		\begin{remark}
			If  \[\max\limits_{i=1,2}\sup\limits_{r\in\mathbb{R}_{+}}\dfrac{|f^{\prime\prime}_{i}(r)|}{[1+(f_{i}^{\prime}(r))^{2}]^{\frac{3}{2}}}<+\infty,\]  
			then $\|\nabla_{\Bt} \Bn\|_{L^{\infty}(\partial\Omega)}$ is bounded, so that the 
			condition in Theorem \ref{th:02} holds.
		\end{remark}
		\begin{remark}\label{re4bound}
			In particular, if $\beta=0$ and
			\[
			\lim\limits_{R\rightarrow \infty}R^{-1}\sup\limits_{z\in[f_{1}(R),\, f_{2}(R)]}|\Bu(R,z)|=0,
			\]
			then  $\Bu$ is trivial. This result generalizes the Liouville-type theorem for the axisymmetric solution in a slab \cite{HWXAHE}.
		\end{remark}
		\begin{remark}
			A particular case of Part (a) of Theorem \ref{th:02} asserts that if	$0\leq \beta<\dfrac{1}{2}$ and $\Bu$ is bounded, then $\Bu$ must be trivial.
		\end{remark}
		\begin{remark}
			Assume that $0\leq \beta<\dfrac{4}{5}$. Let $\Bu$ be a  smooth axisymmetric solution to the Navier-Stokes system \eqref{eqsteadyns} in $\Omega$ supplemented  with Navier boundary conditions \eqref{bounNavboual}. If
			\[
			\int_{\Omega}|\nabla\Bu|^{2}\,d\Bx+2\alpha\int_{\partial\Omega}|\Bu|^{2}\,dS<+\infty,
			\]
			then $\Bu\equiv 0$.
		\end{remark}

		Now we give the key ideas of the proof.  Inspired by the works \cite{aBGWX,aBYA,HWXAHE}, the key idea is,  together with the Bogovskii map, to establish the Saint-Venant type estimate  for the Dirichlet integral of $\Bu$ over a bounded subdomain. The Saint-Venant's principle was first developed in \cite{JKKTARMA66,RATARMA65}  to deal with the solutions for elastic equations.   This idea  was generalized in  \cite{OYASN77} to estimate the growth of Dirichlet integral for  second order elliptic equations, where the uniqueness and existence results of boundary value problems in unbounded domains were obtained. Later on, in order to study the famous Leray problem, i.e., the well-posedness of the boundary value problems for stationary Navier-Stokes system in an infinite long nozzle, the Saint-Venant's principle was used to estimate the Dirichlet integral for  stationary Stokes and Navier-Stokes system \cite{LSZNSL80,PileJSM84}. Furthermore, for the Navier boundary conditions case, the impermeable boundary conditions \eqref{eximform1} together with divergence free property not only gives the Poincar\'{e} inequality
		for $u^{r}$, but also  yields the existence of a Bogovskii map. For the growing layer with unbounded outlets, we choose the truncated domain carefully to get the uniform estimate for the associated Bogovskii map.
		
		The rest of this paper is organized as follows.  In Section \ref{Sec2}, we introduce some notations and collect some elementary lemmas which are used in this paper. The Liouville-type theorems for the axisymmetric flow in a layer with bounded outlets are presented in Section \ref{Sec3}.	  In Section \ref{Sec4},  we study the axisymmetric flow in a growing layer with unbounded outlets.

		\section{Preliminaries}\label{Sec2}
		In this section, we give some preliminaries. First, we introduce the following notations. Assume that $D$ is a bounded domain, define
		\[
		L^p_0(D)=\left\{g: \ \ g\in L^p(D), \ \ \int_D g \, d\Bx =0 \right\}.
		\]
		For any $0\leq a<b\leq\infty$, define $D_{(a,b)}= \left\{(r,z): a<r<b,\ f_{1}(r)<z<f_{2}(r)\right\}$.
		For simplicity, we denote $D_{R}=D_{(R-1, R)}$ for any $R\geq 2$.
		
		Next we introduce the Bogovskii map \cite[Lemma 2.4]{SWX}, which gives a solution to the divergence equation. The proof is due to  Bogovskii \cite{B79DANS}, see also   \cite[Section 2.8]{TT18} and \cite[Section III.3]{GAGP11}.
		\begin{lemma}\label{Bogovskii}
			Let $D \subset \mathbb{R}^{n}$ be a locally Lipschitz domain. Then there exists a constant $C$ such that for any $g\in L^{2}_{0}(D)$, the problem
			\begin{equation}\label{Bogvseq}
				\left\{\begin{aligned}
					&	{\rm div}~\bBV = g& &\text{in}\,\, D,\cr
					&\quad \ \ \bBV=0 & &\text{on}\,\, \partial D \cr
				\end{aligned}
				\right.
			\end{equation}
			has a solution $\bBV\in H^{1}_{0}(D)$ satisfying
			\[
			\|\nabla \bBV\|_{L^{2}(D)}\leq C\|g\|_{L^{2}(D)}.
			\]
			Furthermore, if the domain is of the form
			\[
			D=\bigcup_{k=1}^{N}D_{k},
			\]
			where each $D_{k}$ is star-like with respect to some open ball $B_{k}$ with $\overline{B_{k}}\subset D_{k}$,
			such that  the measure of $\tilde{D}_{i}$ satisfies  $|\tilde{D}_{i}|\neq 0$ with
			$\tilde{D}_{i}=D_{i}\cap\hat{D}_{i}$ and $\hat{D}_{i}=\bigcup_{j=i+1}^{N}D_{j}$,
			then the constant $C$ admits the following estimate
			\begin{equation}\label{eqBoges1}
				C\leq C_{D}\left(\frac{R_{0}}{R}\right)^{n}\left(1+\frac{R_{0}}{R}\right).
			\end{equation}
			Here, $R_{0}$ is the diameter of the domain $D$, $R$ is the smallest radius of the balls $B_{k}$, and
			\begin{equation}\label{eqBoges2es}
				C_{D}=\max_{1\leq k\leq N}\left(1+\frac{|D_{k}|^{\frac{1}{2}}}{|\tilde{D}_{k}|^{\frac{1}{2}}}\right)\prod_{i=1}^{k-1}\left(1+\frac{|\hat{D}_{i}\backslash D_{i}|^{\frac{1}{2}}}{|\tilde{D}_{i}|^{\frac{1}{2}}}\right).
			\end{equation}

		\end{lemma}
		
		Some differential inequalities are used to characterize the growth of local Dirichlet integrals of the nontrivial solutions.
		\begin{lemma}\label{le:differineq}
			Let $y(t)$ be a nondecreasing nonnegative function  and $t_0>1$ be a fixed constant. Suppose that $y(t)$ is not identically zero.
			
			(a)  If $y(t)$  satisfies
			\begin{equation}\label{ineqlemma2-1}
				y(t) \leq C_1 y^{\prime}(t) + C_2 \left[ y^{\prime}(t) \right]^{\frac32} \ \ \ \text{for any}\,\, t\geq t_0,
			\end{equation}
			then
			\begin{equation} \label{qwelemma2-2}
				\varliminf_{t \rightarrow + \infty} t^{-3} y(t)>0.
			\end{equation}

			(b) Assume that $0\leq \delta < \dfrac{4}{5}$. If $y(t)$ satisfies
			\begin{equation}\label{231eqinlemmane2-3}
				y(t) \leq C_3 t^{\delta} y^{\prime}(t)+C_4 t^{\frac{5\delta-1}{2}}\left[ y^{\prime}(t)\right]^{\frac32} \ \ \ \text{for any}\,\,  t\geq  t_0,
			\end{equation}
			then
			\begin{equation}\label{1eqinlemma2-56rn}
				\varliminf_{t \rightarrow + \infty}  t^{5\delta-4}y(t)>0.
			\end{equation}
		\end{lemma}
		\begin{proof}
			The Case $(a)$ and the Case $(b)$ with $\delta=0$ were proved  in \cite[Lemma 2.2]{aBGWX}.
			Now we give the  proof for Case (b) with general $\delta\in\left[0,\, \dfrac{4}{5}\right)$.  Without loss of generality, we  assume that there is a constant $t_{0}\geq 1$, such that $y(t_{0})>0$. So one has $y^{\prime}(t)>0$, for all $t>t_{0}$. This yields  $y(t)>y(t_{0})$.
			By Young's inequality, one has
			\[
			C_{3}t^{\delta}y^{\prime}(t)\leq \frac{1}{2}y(t_{0})+Ct^{\frac{3}{2}\delta}[y^{\prime}(t)]^{\frac{3}{2}}.
			\]
			This, together with \eqref{231eqinlemmane2-3},  yields
			\begin{equation}\label{diffinle22}
				y(t) \leq Ct^{\frac{3}{2}\delta} \left[y^{\prime}(t)\right]^{\frac{3}{2}}+Ct^{\frac{5\delta-1}{2}}\left[y^{\prime}(t)\right]^{\frac{3}{2}}
				\ \ \ \text{for any}\,\,  t\geq  t_0.
			\end{equation}
			
			(i) For $\delta\in\left[0,\, \dfrac{1}{2} \right)$, it follows from \eqref{diffinle22} that one obtains
			\begin{equation}\label{eqdeifflew23}
				y(t) \leq Ct^{\frac{3}{2}\delta} \left[y^{\prime}(t)\right]^{\frac{3}{2}} \ \ \ \text{for any}\,\,  t\geq  t_0.
			\end{equation}
			Integrating from $t_{0}$ to $t$ yields
			\[
			t^{3-3\delta}\leq Cy(t).
			\]
			This, together with \eqref{eqdeifflew23}, yields 
			\begin{equation}\label{eqredwes25}
				t^{2-3\delta}\leq Cy^{\prime}(t).
			\end{equation}
			Combining \eqref{231eqinlemmane2-3} and \eqref{eqredwes25}, one arrives at
			\[
			\begin{split}
				y(t)&\leq
				Ct^{\delta} y^{\prime}(t)+C t^{\frac{5\delta-1}{2}}\left[ y^{\prime}(t)\right]^{\frac{3}{2}}\\
				&\leq
				Ct^{\delta} \left[ y^{\prime}(t)\right]^{\frac{3}{2}}t^{\frac{3}{2}\delta-1}+C t^{\frac{5\delta-1}{2}}\left[ y^{\prime}(t)\right]^{\frac{3}{2}}\\
				&\leq
				Ct^{\frac{5\delta-1}{2}}\left[ y^{\prime}(t)\right]^{\frac{3}{2}}.
			\end{split}
			\]
			
			(ii)  For $\delta\in\left[\dfrac{1}{2},\, \dfrac{4}{5} \right)$, since $y(t)$ is nondecreasing, it holds that
			\begin{equation}\label{eqdiffeq27t}
				y(t) \leq Ct^{\frac{5\delta-1}{2}} \left[y^{\prime}(t)\right]^{\frac{3}{2}} \ \ \ \text{for any}\,\,  t\geq  t_0.
			\end{equation}
			Next, we integrate \eqref{eqdiffeq27t} from $t_{0}$ to $t$, yields
			\[
			t^{5\delta-4}y(t)\geq C>0 \ \ \ \text{for any}\,\,  t\geq  2t_0.
			\]
			This completes the proof for Case (b) of Lemma \ref{le:differineq}.
		\end{proof}
		Now we  introduce a Sobolev embedding inequality for functions with no-slip boundary conditions in a two-dimensional channel. The similar Sobolev embedding inequality can be found in \cite[Appendix]{LSZNSL80}. 
		\begin{lemma}\label{L4ineqnoslipes}
			For any $\Bu\in H^{1}(D_{(a,b)})$ with $\Bu=0$ on $\partial D_{(a,b)}\cap \partial D_{(0, \infty)}$, one has
			\begin{equation}\label{Poincare1}
				\|\Bu\|_{L^{2}(D_{(a,b)})}
				\leq M_{1}(D_{(a,b)})\|\nabla   \Bu\|_{L^{2}(D_{(a,b)})}
			\end{equation}
			and
			\begin{equation}\label{Poincare2Sobole4}
				\|\Bu\|_{L^{4}(D_{(a,b)})}
				\leq M_{2}(D_{(a,b)})\|\nabla   \Bu\|_{L^{2}(D_{(a,b)})},
			\end{equation}
			where
			\begin{equation}\label{Poincons1cha}
				M_{1}(D_{(a,b)})=\sup\limits_{r\in (a, b)}f(r),
			\end{equation}
			\begin{equation}\label{eqconMOFa}
				\begin{split}
					M_{2}(D_{(a,b)})
					=C\left\{|D_{(a,b)}|^{\frac{1}{2}}
					\left[(b-a)^{-1}M_{1}(D_{(a,b)})+1
					\right]\right\}^{\frac{1}{2}}
				\end{split}
			\end{equation}
			and $C$ is a universal constant.
		\end{lemma}
		\begin{proof}
			Note that \eqref{Poincare1} is exactly the  Poincar\'e inequality. Now we prove the Sobolev embedding inequality \eqref{Poincare2Sobole4}. First,  we  extend $\Bu$ as follows 
			\[
			\begin{aligned}
				\tilde{\Bu}&=\left\{
				\begin{aligned}
					&\Bu\ \ && \Bx\in D_{(a,b)},\\
					&0\ \ &&\Bx\in \left((a,b)\times \mathbb{R}\right)\setminus D_{(a,b)}. 
				\end{aligned}
				\right.
			\end{aligned}
			\]
			Clearly, $\tilde{\Bu} \in H^{1}((a,b)\times \mathbb{R})$.
			By Gagliardo-Nirenberg interpolation inequality (cf. \cite[Theorem 1.1]{LZCPAA22}), one derives
			\begin{equation}\label{Gninextfu}
				\begin{split}
					\int_{D_{(a,b)}}|\Bu|^{4}\,d\Bx
					&=\int_{a}^{b}\,dr\int_{-\infty}^{+\infty}|\tilde{\Bu}|^{4}\, dz\\
					&\leq
					C\int_{a}^{b}\,dr
					\int_{-\infty}^{+\infty}|\partial_{z}\tilde{\Bu}|^{2}\, dz
					\left(\int_{-\infty}^{+\infty}|\tilde{\Bu}|\, dz\right)^{2}\\
					&\leq C\left(\sup_{r\in(a,b)}\int_{-\infty}^{+\infty}|\tilde{\Bu}|\,dz\right)^{2}\int_{D_{(a,b)}}
					|\partial_{z}\tilde{\Bu}|^{2}\, d\Bx.
				\end{split}
			\end{equation}
			
			On the other hand, for any $r\in (a,b)$, one has
			\[
			|\tilde{\Bu}(r, z)|\leq |\tilde{\Bu}(r^{0},z)|+\int_{a}^{b}
			|\partial_{r}\tilde{\Bu}|\,dr.
			\]
			Then we integrate the inequality with respect to $z$ and $r^{0}$ from $-\infty$ to $+\infty$ and $a$ to $b$, respectively, and get
			\begin{equation}\label{eqvintgnq}
				\begin{split}
					\int_{-\infty}^{+\infty}|\tilde{\Bu}|
					\,dz
					&\leq (b-a)^{-1}\int_{-\infty}^{+\infty}\int_{a}^{b}|\tilde{\Bu}|\,drdz +\int_{-\infty}^{+\infty}\int_{a}^{b}|\partial_{r}\tilde{\Bu}|\,drdz \\[3pt]
					&\leq (b-a)^{-1}\int_{D_{(a,b)}}|\Bu|\,d\Bx +\int_{D_{(a,b)}}|\nabla\Bu|\,d\Bx \\[3pt]
					&\leq
					|D_{(a,b)}|^{\frac{1}{2}}\left[(b-a)^{-1}\left(\int_{D_{(a,b)}}|\Bu|^{2}\,d\Bx\right)^{\frac{1}{2}} +\left(\int_{D_{(a,b)}}|\nabla\Bu|^{2}\,d\Bx\right)^{\frac{1}{2}}
					\right] \\[3pt]
					&\leq |D_{(a,b)}|^{\frac{1}{2}}\left[(b-a)^{-1}M_{1}(D_{(a,b)})+1
					\right]\left(\int_{D_{(a,b)}}|\nabla\Bu|^{2}\,d\Bx\right)^{\frac{1}{2}},
				\end{split}
			\end{equation}
			where the fourth inequality is due to the Poincar\'e inequality \eqref{Poincare1}.
			From  \eqref{Gninextfu} and \eqref{eqvintgnq}, one obtains
			\[
			\|\Bu\|_{L^{4}(D_{(a,b)})}\leq M_{2}(D_{(a,b)})\|\nabla \Bu\|_{L^{2}(D_{(a,b)})},
			\]
			where $M_{2}(D_{(a,b)})$ is given in \eqref{eqconMOFa}. This completes the proof of Lemma \ref{L4ineqnoslipes}.
		\end{proof}
		
		A similar Sobolev embedding inequality for functions with Navier boundary conditions is as follows, whose proof can be referred to \cite[Lemma 2.2]{SWX}.
		\begin{lemma}\label{1Nabineqnoslipes}
			Suppose that $f(r)\geq d_{a,b}>0$ for any $r\in (a,b)$, then for any $\Bu\in H^{1}(D_{(a,b)})$ satisfying
			$\Bu\cdot\Bn=0$  on $\partial D_{(a,b)}\cap \partial D_{(0, \infty)}$, one has the corresponding Poincar\'e inequality
			\begin{equation}\label{Poincare3}
				\|\Bu\|_{L^{2}(D_{(a,b)})}
				\leq M_{1}\left(\|\nabla   \Bu\|_{L^{2}(D_{(a,b)})}+\|  \Bu\|_{L^{2}(\partial D_{(a,b)}\cap \partial D_{(0, \infty)})}\right)
			\end{equation}
			and	the Sobolev embedding inequality
			\begin{equation}\label{Poincare4}
				\|\Bu\|_{L^{4}(D_{(a,b)})}\leq M_{3}(D_{(a,b)})(\|\nabla \Bu\|_{L^{2}(D_{(a,b)})}+\| \Bu\|_{L^{2}(\partial D_{(a,b)}\cap \partial D_{(0, \infty)})}),
			\end{equation}
			where
			\begin{equation}\label{123eqconMOFa}
				\begin{split}
					M_{3}(D_{(a,b)})
					=C(1+\|(f^{\prime}_{1},f^{\prime}_{2})\|_{L^{\infty}(D_{(a,b)})})
					\left(\frac{M_{1}}{b-a}+1\right)^{\frac{1}{2}}
					\left[|D_{(a,b)}|+ (b-a)d_{a,b}\right]^{\frac{1}{4}}\left(1+\frac{M_{1}}{d_{a,b}}\right)
				\end{split}
			\end{equation}
			with a universal constant $C$ and $M_{1}=M_{1}(D_{(a,b)})$ defined in \eqref{Poincons1cha}.
		\end{lemma}
		\section{Flows in a layer with bounded outlets}\label{Sec3}
		In order  to illustrate our main idea in a clear way, we would like to first give the proof for the Liouville-type theorems about the flows in a bounded layer, i.e., $\beta=0$ in \eqref{growboundabe1}. In the next section, we give the major difference for the proof on the Liouville-type theorems for the flows in a growing layer.		Let's introduce a  cut-off function $\varphi_R(r)$,
		\be \label{cut-off}
		\varphi_R(r)
		= \left\{ \ba
		&1,\ \ \ \ \ \ \ \ \ \ r < R-1, \\
		&R-r,\ \ \ \ R-1 \leq r \leq R, \\
		&0, \ \ \ \ \ \ \ \ \ \ r > R
		\ea  \right.
		\ee
		and the truncated  domain  
		\begin{equation}\label{truncatdomOR}
			\OR =\left\{
			\Bx \in\mathbb{R}^{3} :
			(x_{1}, x_{2})\in B_{R} \setminus \overline{B_{R-1}},\,
			f_{1}(r)<x_{3}<f_{2}(r)  \right\}, 
		\end{equation}
		where
		$r=\sqrt{x_{1}^{2}+x_{2}^{2}}$,
		$B_{R}=\left\{(x_{1}, x_{2}):x_1^2+x_2^2 <R^{2} \right\}$.
		
		\subsection{No-slip boundary conditions case}\label{Sec31}
		In this subsection, we assume that the Navier-Stokes system is supplemented with no-slip boundary  conditions \eqref{curboundarynoslip}.
		
		\begin{proof}[Proof  of Theorem \ref{th:01} (In the case $\beta=0$)]
			
			The proof is divided into three steps.
			
			\emph{Step 1.} {\emph{Set up.}} Assume that $\Bu$ is a smooth solution to \eqref{eqsteadyns}
			in $\Omega$ with no-slip boundary conditions. Multiplying the first equation in \eqref{eqsteadyns}  by $\varphi_{R}(r)\Bu$ and integrating by parts, one obtains
			\begin{equation}\label{mueqine}
				\int_{\Omega}|\nabla\Bu|^{2}\varphi_{R}\, d\Bx
				=-\int_{\Omega}\nabla\varphi_{R}\cdot \nabla\Bu\cdot \Bu\, d\Bx+\int_{\Omega}\frac{1}{2}|\Bu|^{2}\Bu\cdot \nabla\varphi_{R}\, d\Bx+\int_{\Omega}P\Bu\cdot \nabla \varphi_{R}\, d\Bx.
			\end{equation}
			The third term on the right hand can be written as
			\begin{equation}\label{eqA83}
				\int_{\Omega}P\Bu\cdot \nabla\varphi_{R}\, d\Bx
				=-2\pi \int_{R-1}^{R}\int_{f_{1}(r)}^{f_{2}(r)}Pu^{r}r\,dzdr.
			\end{equation}	
			The divergence equation in \eqref{axieqsteadyns} can  be written as
			\[
			\partial_{r}(ru^{r})+\partial_{z}(ru^{z})=0.
			\]
			Taking the no-slip boundary conditions \eqref{curboundarynoslip} into account, one can obtain
			\[
			\begin{split}
				\partial_{r}\int_{f_{1}(r)}^{f_{2}(r)}ru^{r}dz
				&=\int_{f_{1}(r)}^{f_{2}(r)}\partial_{r}(ru^{r})\,dz+ru^{r}(r,f_{2}(r))f_{2}^{\prime}(r)-ru^{r}(r,f_{1}(r))f_{1}^{\prime}(r)\\
				&=-\int_{f_{1}(r)}^{f_{2}(r)}\partial_{z}(ru^{z})\,dz=0.
			\end{split}
			\]
			Hence,
			\begin{equation}\label{eqA85}
				\int_{f_{1}(r)}^{f_{2}(r)}ru^{r}\,dz=0
				\quad \text{and}\quad \int_{R-1}^{R}\int_{f_{1}(r)}^{f_{2}(r)}ru^{r}\,dzdr=0.	\end{equation}
			By virtue of \eqref{eqA85} and Lemma \ref{Bogovskii},
			there exists a vector valued function $\Ps_{R}(r,z)\in H^{1}_{0}(D_{R};\, \mathbb{R}^{2})$ satisfying
			\begin{equation}\label{eqA86}
				\partial_{r}\Psi_{R}^{r}+\partial_{z}\Psi_{R}^{z}=ru^{r}\ \ \ \ \mbox{in}\ D_R
			\end{equation}
			and
			\begin{equation}\label{eqA87}
				\|\partial_{r}\Ps_{R}\|_{L^{2}(D_{R})}+	\|\partial_{z}\Ps_{R}\|_{L^{2}(D_{R})}\leq C\|ru^{r}\|_{L^{2}(D_{R})}\leq
				CR\|\nabla\Bu\|_{L^{2}(D_{R})}\leq
				CR^{\frac{1}{2}}\|\nabla\Bu\|_{L^{2}(\OR)}.
			\end{equation}
			Therefore, combining \eqref{eqA83} and \eqref{eqA86}  one derives
			\begin{equation}\label{eqA88}
				\begin{split}
					\int_{\Omega}P\Bu\cdot \nabla \varphi_{R}\,d\Bx
					=&-2\pi \int_{R-1}^{R}\int_{f_{1}(r)}^{f_{2}(r)}
					P(\partial_{r}\Psi_{R}^{r}+\partial_{z}\Psi_{R}^{z})\,dzdr\\
					=&\,2\pi
					\int_{R-1}^{R}\int_{f_{1}(r)}^{f_{2}(r)}(\partial_{r}P\Psi_{R}^{r}+\partial_{z}P\Psi_{R}^{z})\,dzdr.
				\end{split}
			\end{equation}
			From the cylindrical form \eqref{axieqsteadyns} for the axisymmetric flow, the gradient $(\partial_{r}P,\, \partial_{z}P)$ of the pressure satisfies
			\begin{equation}\label{eqA89}
				\begin{cases}
					(u^{r}\partial_{r}+u^{z}\partial_{z})u^{r}-\dfrac{(u^{\theta})^2}{r}+\partial_{r}P=\left(\partial_{r}^{2}+\dfrac{1}{r}\partial_{r}+\partial_{z}^{2}-\dfrac{1}{r^2}\right)u^{r},\\[8pt]
					(u^{r}\partial_{r}+u^{z}\partial_{z})u^{z}+\partial_{z}P=\left(\partial_{r}^{2}+\dfrac{1}{r}\partial_{r}+\partial_{z}^{2}\right)u^{z}.
				\end{cases}
			\end{equation}
			Using the system \eqref{eqA89} and integration by parts, one obtains
			\begin{equation}\label{eqparrP}
				\begin{split}
					&\int_{R-1}^{R}\int_{f_{1}(r)}^{f_{2}(r)}\partial_{r}P\Psi^{r}_{R}\,dzdr\\
					=& \int_{R-1}^{R}\int_{f_{1}(r)}^{f_{2}(r)}\left[\left(\partial_{r}^{2}+\dfrac{1}{r}\partial_{r}+\partial_{z}^{2}-\dfrac{1}{r^2}\right)u^{r}-(u^{r}\partial_{r}+u^{z}\partial_{z})u^{r}+\dfrac{(u^{\theta})^2}{r}\right]\Psi^{r}_{R}\,dzdr
					\\
					=&
					-\int_{R-1}^{R}\int_{f_{1}(r)}^{f_{2}(r)}(\partial_{r}u^{r}\partial_{r}\Psi^{r}_{R}+\partial_{z}u^{r}\partial_{z}\Psi^{r}_{R})\,dzdr+\int_{R-1}^{R}\int_{f_{1}(r)}^{f_{2}(r)}\left[\left(\dfrac{1}{r}\partial_{r}-\dfrac{1}{r^2}\right)u^{r}\right]\Psi^{r}_{R}\,dzdr\\
					&-\int_{R-1}^{R}\int_{f_{1}(r)}^{f_{2}(r)}\left[(u^{r}\partial_{r}+u^{z}\partial_{z})u^{r}-\dfrac{(u^{\theta})^2}{r}\right]\Psi^{r}_{R}\,dzdr
				\end{split}
			\end{equation}
			and
			\begin{equation}\label{eqparpazP}
				\begin{split}
					&\int_{R-1}^{R}\int_{f_{1}(r)}^{f_{2}(r)}\partial_{z}P\Psi^{z}_{R}\,dzdr\\
					=&-\int_{R-1}^{R}\int_{f_{1}(r)}^{f_{2}(r)}(\partial_{r}u^{z}\partial_{r}\Psi^{z}_{R}+\partial_{z}u^{z}\partial_{z}\Psi^{z}_{R})\,dzdr+\int_{R-1}^{R}\int_{f_{1}(r)}^{f_{2}(r)}\left(\dfrac{1}{r}\partial_{r}u^{z}\right)\Psi^{z}_{R}\,dzdr\\
					&-\int_{R-1}^{R}\int_{f_{1}(r)}^{f_{2}(r)}\left[(u^{r}\partial_{r}+u^{z}\partial_{z})u^{z}\right]\Psi^{z}_{R}\,dzdr.
				\end{split}
			\end{equation}
			
			\emph{Step 2.} {\emph{Proof  for Case (b) of Theorem \ref{th:01} ($\beta=0$)}.}
			Now we are ready to estimate the first two terms on the right hand sides of \eqref{mueqine} and  \eqref{eqparrP}--\eqref{eqparpazP}.
			By H\"older inequality and Lemma \ref{L4ineqnoslipes},
			one has
			\begin{equation}\label{eqA93}
				\begin{split}
					\left|\int_{\Omega}\nabla\varphi_{R}\cdot \nabla\Bu\cdot\Bu\, d\Bx \right|&\leq 2\pi
					\left|\int_{R-1}^{R}\int_{f_{1}(r)}^{f_{2}(r)}|\nabla\Bu|\cdot|\Bu|r\,dzdr\right|\\ 
					&\leq CR\|\nabla\Bu\|_{L^{2}(D_{R})}\|\Bu\|_{L^{2}(D_{R})}\\
					&\leq CR\|\nabla\Bu\|_{L^{2}(D_{R})}^{2}\leq C\|\nabla\Bu\|^{2}_{L^{2}(\OR)}
				\end{split}
			\end{equation}
			and	
			\begin{equation}\label{eqA92}
				\begin{split}
					\left|\int_{\Omega}\frac{1}{2}\left|\Bu \right|^2\Bu\cdot \nabla \varphi_{R}\,d\Bx\right|
					&=\pi
					\left|\int_{R-1}^{R}\int_{f_{1}(r)}^{f_{2}(r)}|\Bu|^{2}u^{r}r\,dzdr\right|
					\\
					&\leq
					CR\|\Bu\|^{2}_{L^{4}(D_{R})}\cdot \|u^{r}\|_{L^{2}(D_{R})} \\
					&\leq
					CR\cdot R^{-1}\|\nabla\Bu\|^{2}_{L^{2}(\OR)}\cdot R^{-\frac{1}{2}}\|\nabla\Bu\|_{L^{2}(\OR)}\\
					&\leq CR^{-\frac{1}{2}}\|\nabla\Bu\|^{3}_{L^{2}(\OR)}.
				\end{split}		
			\end{equation}
			As for the terms on right hand of \eqref{eqparrP}, by \eqref{eqA86}--\eqref{eqA87} and Lemma \ref{L4ineqnoslipes} yield
			\begin{equation}\label{eqA94}
				\begin{split}
					\left| \int_{R-1}^{R}\int_{f_{1}(r)}^{f_{2}(r)}(\partial_{r}u^{r}\partial_{r}\Psi^{r}_{R}+\partial_{z}u^{r}\partial_{z}\Psi^{r}_{R})\,dzdr\right|
					\leq& C\|(\partial_{r}, \partial_{z} )\Bu\|_{L^{2}(D_{R})}	\cdot\|(\partial_{r}, \partial_{z} )\Psi^{r}_{R}\|_{L^{2}(D_{R})}\\
					\leq& CR^{-\frac{1}{2}}\|\nabla\Bu\|_{L^{2}(\OR)}\cdot	R^{\frac{1}{2}}\|\nabla\Bu\|_{L^{2}(\OR)}\\
					\leq& C\|\nabla\Bu\|^{2}_{L^{2}(\OR)}
				\end{split}
			\end{equation}
			and
			\begin{equation}\label{eqA95}
				\begin{split}
					\left|\int_{R-1}^{R}\int_{f_{1}(r)}^{f_{2}(r)}\left[\left(\dfrac{1}{r}\partial_{r}-\dfrac{1}{r^2}\right)u^{r}\right]\Psi^{r}_{R}\,dzdr\right|
					\leq&C\left(R^{-1}+R^{-2}\right)\|\nabla\Bu\|_{L^{2}(D_{R})}\cdot\|\Psi^{r}_{R}\|_{L^{2}(D_{R})}\\
					\leq&C\left(R^{-1}+ R^{-2}\right)\|\nabla\Bu\|_{L^{2}(D_{R})}\cdot\| \partial_{z} \Psi^{r}_{R}\|_{L^{2}(D_{R})}\\
					\leq& C\left(R^{-\frac{3}{2}}+ R^{-\frac{5}{2}}\right)\|\nabla\Bu\|_{L^{2}(\OR)}\cdot  R^{\frac{1}{2}}\|\nabla\Bu\|_{L^{2}(\OR)}\\
					\leq&C\left(R^{-1}+ R^{-2}\right)\| \nabla\Bu\|^{2}_{L^{2}(\OR)}.
				\end{split}
			\end{equation}
			Furthermore, it holds that
			\begin{equation}\label{eqA96}
				\begin{split}
					\left|\int_{R-1}^{R}\int_{f_{1}(r)}^{f_{2}(r)}\left[(u^{r}\partial_{r}+u^{z}\partial_{z})u^{r}\right]\Psi^{r}_{R}\,dzdr\right|
					\leq& C\|\Bu\|_{L^{4}(D_{R})}\cdot\|(\partial_{r}, \partial_{z})u^{r}\|_{L^{2}(D_{R})}\cdot\|\Psi^{r}_{R}\|_{L^{4}(D_{R})}\\
					\leq&
					C\|\nabla\Bu\|_{L^{2}(D_{R})}\cdot R^{-\frac{1}{2}}\|\nabla\Bu\|_{L^{2}(\OR)}\cdot \|\nabla\Psi^{r}_{R}\|_{L^{2}(D_{R})}\\
					\leq&
					CR^{-\frac{1}{2}}\|\nabla\Bu\|_{L^{2}(\OR)} R^{-\frac{1}{2}}\|\nabla\Bu\|_{L^{2}(\OR)}  R^{\frac{1}{2}}\|\nabla\Bu\|_{L^{2}(\OR)}\\
					\leq&
					CR^{-\frac{1}{2}}\|\nabla\Bu\|^{3}_{L^{2}(\OR)}
				\end{split}
			\end{equation}
			and
			\begin{equation}\label{eqestiA332}
				\begin{split}
					\left|\int_{R-1}^{R}\int_{f_{1}(r)}^{f_{2}(r)}\left[\dfrac{(u^{\theta})^{2}}{r}\right]\Psi^{r}_{R}\,dzdr\right|
					\leq& CR^{-1}\|\Bu\|^{2}_{L^{4}(D_{R})}\cdot\|\Psi^{r}_{R}\|_{L^{2}(D_{R})}\\ 
					\leq&
					CR^{-2}\|\nabla\Bu\|^{2}_{L^{2}(\OR)}\cdot R^{\frac{1}{2}}\|\nabla\Bu\|_{L^{2}(\OR)}\\	
					\leq&
					CR^{-\frac{3}{2}}\|\nabla\Bu\|^{3}_{L^{2}(\OR)}.
				\end{split}
			\end{equation}
			Combining the estimates \eqref{eqA94}--\eqref{eqestiA332}, one derives
			\begin{equation}\label{eqA97}
				\left|\int_{R-1}^{R}\int_{f_{1}(r)}^{f_{2}(r)}\partial_{r}P\Psi^{r}_{R}\,dzdr\right|\leq C\|\nabla\Bu\|^{2}_{L^{2}(\OR)}+CR^{-\frac{1}{2}}\|\nabla\Bu\|^{3}_{L^{2}(\OR)}.
			\end{equation}
			Similarly, one has
			\begin{equation}\label{eqA98}
				\left|\int_{R-1}^{R}\int_{f_{1}(r)}^{f_{2}(r)}\partial_{z}P\Psi^{z}_{R}\,dzdr\right|\leq C\|\nabla\Bu\|^{2}_{L^{2}(\OR)}+CR^{-\frac{1}{2}}\|\nabla\Bu\|^{3}_{L^{2}(\OR)}.
			\end{equation}
			Combining the estimates  \eqref{eqA93}--\eqref{eqA92} and
			\eqref{eqA97}--\eqref{eqA98}, one obtains that
			\begin{equation}\label{eqA99}
				\int_{\Omega}|\nabla \Bu|^{2}\varphi_{R}\,d\Bx\leq C\|\nabla\Bu\|^{2}_{L^{2}(\OR)}+CR^{-\frac{1}{2}}\|\nabla\Bu\|^{3}_{L^{2}(\OR)}.
			\end{equation}
			
			Let
			\begin{equation}\label{11eqYRA100}
				Y(R)=\int_{\Omega}|\nabla\Bu|^{2}\varphi_{R}\,d\Bx.
			\end{equation}
			The explicit form \eqref{cut-off} of $\varphi_{R}(r)$ gives
			\[
			Y(R)
			=2\pi\left(\int_{0}^{R-1}\int_{f_{1}(r)}^{f_{2}(r)}|\nabla\Bu(r,z)|^{2}r\,dzdr+\int_{R-1}^{R}\int_{f_{1}(r)}^{f_{2}(r)}|\nabla\Bu(r,z)|^{2}(R-r)r\,dzdr\right)
			\]
			and
			\[
			Y^{\prime}(R)
			=2\pi \int_{R-1}^{R}\int_{f_{1}(r)}^{f_{2}(r)}|\nabla\Bu(r,z)|^{2}r\,dzdr
			=\int_{\OR}|\nabla\Bu|^{2}\,d\Bx.
			\]
			Hence the estimate \eqref{eqA99} can be written as
			\begin{equation}\label{eqA103}
				Y(R)\leq CY^{\prime}(R)+CR^{-\frac{1}{2}} [Y^{\prime}(R)]^{\frac{3}{2}}.
			\end{equation}
			It follows from  Lemma \ref{le:differineq} (b)
			that if $Y(R)$ is not identically zero, then
			\[
			\varliminf_{R\rightarrow +\infty}R^{-4}    Y(R)>0.
			\]
			Note that $\displaystyle\varliminf_{R \rightarrow + \infty} R^{-4} Y(R)=0$,  therefore $Y(R)$ must be identically zero,
			and so $\nabla\Bu\equiv0$. This, together with  the no-slip boundary conditions  finishes the proof   for Case (b) of Theorem \ref{th:01} ($\beta=0$).
			
			\emph{Step 3.} {\emph{Proof  for Case (a) of Theorem \ref{th:01} ($\beta=0$)}.} Next  we estimate the  terms on the right hand side of \eqref{mueqine} in a different way. Using H\"older inequality and Poincar\'{e} inequality \eqref{Poincare1} yield
			\begin{equation}\label{eqA104}
				\begin{split}
					\left|\int_{\Omega}\nabla\varphi_{R}\cdot \nabla\Bu\cdot\Bu\,d\Bx \right|
					&\leq C\|\nabla\Bu\|_{L^{2}(\OR)}\|\Bu\|_{L^{2}(\OR)}\\
					&\leq CR^{\frac{1}{2}}\|\Bu\|_{L^{\infty}(\OR)}\|\nabla\Bu\|_{L^{2}(\OR)}
				\end{split}
			\end{equation}
			and
			\begin{equation}\label{eqA105}
				\begin{split}
					\left|\int_{\Omega}\frac{1}{2}\left|\Bu \right|^2\Bu\cdot \nabla \varphi_{R}\,d\Bx\right|
					&=
					\pi\left|\int_{R-1}^{R}\int_{f_{1}(r)}^{f_{2}(r)}|\Bu|^{2}u^{r}r\,dzdr\right| \\
					&\leq CR\|\Bu\|^{2}_{L^{\infty}(D_{R})}\|u^{r}\|_{L^{2}(D_{R})}\\
					&\leq
					CR^{\frac{1}{2}}\|\Bu\|^{2}_{L^{\infty}(\OR)}\|\nabla\Bu\|_{L^{2}(\OR)}.
				\end{split}
			\end{equation}
			As for the terms on right hand of \eqref{eqparrP}, by \eqref{eqA86}--\eqref{eqA87} and Poincar\'{e} inequality \eqref{Poincare1}, one has
			\begin{equation}\label{eqA106}
				\begin{split}
					&\left| \int_{R-1}^{R}\int_{f_{1}(r)}^{f_{2}(r)}(\partial_{r}u^{r}\partial_{r}\Psi^{r}_{R}+\partial_{z}u^{r}\partial_{z}\Psi^{r}_{R})\,dzdr\right|\\
					\leq&  C\|(\partial_{r},\partial_{z})u^{r}\|_{L^{2}(D_{R})}	\|(\partial_{r},\partial_{z})\Psi^{r}_{R}\|_{L^{2}(D_{R})}\\
					\leq&  CR^{-\frac{1}{2}}\|\nabla\Bu\|_{L^{2}(\OR)}\cdot\|ru^{r}\|_{L^{2}(D_{R})}\\
					\leq& CR^{-\frac{1}{2}}\|\nabla\Bu\|_{L^{2}(\OR)}\cdot R^{\frac{1}{2}}\|u^{r}\|_{L^{2}(\OR)}\\
					\leq&
					CR^{\frac{1}{2}} \|\nabla\Bu\|_{L^{2}(\OR)}\|\Bu\|_{L^{\infty}(\OR)}
				\end{split}
			\end{equation}
			and
			\begin{equation}\label{eqA107}
				\begin{split}
					&\left|\int_{R-1}^{R}\int_{f_{1}(r)}^{f_{2}(r)}\left[\left(\dfrac{1}{r}\partial_{r}-\dfrac{1}{r^2}\right)u^{r}\right]\Psi^{r}_{R}\,dzdr\right|\\
					\leq&
					C\left(R^{-1}\| \nabla \Bu\|_{L^{2}(D_{R})}+R^{-2}\| u^{r}\|_{L^{2}(D_{R})}\right)\|\Psi^{r}_{R}\|_{L^{2}(D_{R})}\\
					\leq&
					CR^{-1}\cdot R^{-\frac{1}{2}}\| \nabla\Bu\|_{L^{2}(\OR)}
					\cdot
					R^{\frac{1}{2}}\|u^{r}\|_{L^{2}(\OR)}\\
					\leq&
					CR^{-\frac{1}{2}}\|\nabla \Bu\|_{L^{2}(\OR)}
					\|\Bu\|_{L^{\infty}(\OR)}.
				\end{split}
			\end{equation}
			Similarly, it holds that
			\begin{equation}\label{eqA108}
				\begin{split}
					&\left|\int_{R-1}^{R}\int_{f_{1}(r)}^{f_{2}(r)}\left[(u^{r}\partial_{r}+u^{z}\partial_{z})u^{r}-\dfrac{(u^{\theta})^2}{r}\right]\Psi^{r}_{R}\,dzdr\right|\\	
					\leq& C\|\Bu\|_{L^{\infty}{(D_{R})}}\left(\|(\partial_{r},\partial_{z})u^{r}\|_{L^{2}(D_{R})}+R^{-1}\|u^{\theta}\|_{L^{2}(D_{R})}\right)\|\Psi^{r}_{R}\|_{L^{2}(D_{R})}\\
					\leq&
					C\|\Bu\|_{L^{\infty}(\OR)}\left(R^{-\frac{1}{2}}\|\nabla\Bu\|_{L^{2}(\OR)}+R^{-1}\|u^{\theta}\|_{L^{\infty}(\OR)}\right)\cdot R^{\frac{1}{2}}\|u^{r}\|_{L^{2}(\OR)}\\
					\leq&
					CR^{\frac{1}{2}}\|\nabla\Bu\|_{L^{2}(\OR)}	\|\Bu\|^{2}_{L^{\infty}(\OR)}.
				\end{split}
			\end{equation}
			Collecting	 the estimates \eqref{eqA106}--\eqref{eqA108}, one arrives at
			\begin{equation}\label{eqrpestiq1}
				\left|\int_{R-1}^{R}\int_{f_{1}(r)}^{f_{2}(r)}\partial_{r}P\Psi^{r}_{R}\,dzdr\right|
				\leq CR^{\frac{1}{2}}(\|\Bu\|_{L^{\infty}(\OR)}+\|\Bu\|^{2}_{L^{\infty}(\OR)})\|\nabla\Bu\|_{L^{2}(\OR)}.
			\end{equation}
			Similarly, one has
			\begin{equation}\label{eqA110}
				\left|\int_{R-1}^{R}\int_{f_{1}(r)}^{f_{2}(r)}\partial_{z}P\Psi^{z}_{R}\,dzdr\right|
				\leq CR^{\frac{1}{2}}(\|\Bu\|_{L^{\infty}(\OR)}+\|\Bu\|^{2}_{L^{\infty}(\OR)})\|\nabla\Bu\|_{L^{2}(\OR)}.
			\end{equation}
			Therefore, it can be shown that
			\begin{equation}\label{diineqA111}
				Y(R)\leq CR^{\frac{1}{2}}(\|\Bu\|_{L^{\infty}(\OR)}+\|\Bu\|^{2}_{L^{\infty}(\OR)})[Y^{\prime}(R)]^{\frac{1}{2}},		
			\end{equation}
			where $Y(R)$ is defined in \eqref{11eqYRA100}.		
			
			Suppose $\Bu$ is not identically equal to zero and $\Bu$ satisfies
			\[
			\lim\limits_{R\rightarrow +\infty}R^{-1}\sup\limits_{z\in[f_{1}(R),\,\, f_{2}(R)]}|\Bu(R,z)|=0.
			\]
			For any small $\epsilon>0$, there exists a constant $R_{0}(\epsilon)>2/\epsilon$
			such that
			\[
			\|\Bu\|_{L^{\infty}(\OR)}\leq\epsilon R \qquad \text{for any} \ R\geq
			R_{0}(\epsilon).
			\]
			Since $Y(R)>0$,	 the inequality \eqref{diineqA111} implies that
			\begin{equation}\label{eqA112}
				(C\epsilon)^{-2}R^{-5}\leq\dfrac{Y^{\prime}(R)}{[Y(R)]^{2}}.
			\end{equation}
			If $\Bu$ is not equal to zero, according to   Case (b) of Theorem \ref{th:01} ($\beta=0$), $Y(R)$ must be unbounded as $R\rightarrow +\infty$. For  $R$ sufficiently large, integrating \eqref{eqA112} over $[R,+\infty) $ one arrives at
			\[
			R^{-4}Y(R)\leq 4(C\epsilon)^{2}.
			\]
			Since $\epsilon$ can be arbitrarily small, this implies  Case (b) of Theorem \ref{th:01} ($\beta=0$), and leads to a contradiction with the assumption that $\Bu$ is not identically zero. This finishes the proof of Theorem \ref{th:01} with $\beta=0$.
		\end{proof}	
		\subsection{Navier boundary conditions case} \label{Sec32}
		This subsection devotes to the study on the Navier-Stokes system   with Navier boundary  conditions \eqref{bounNavboual}.
		
		\begin{proof}[Proof  of Theorem \ref{th:02} (In the case $\beta=0$)]
			
			The proof is divided into three steps.
			
			\emph{Step 1.} {\emph{Set up.}} Assume that $\Bu$ is a smooth solution to \eqref{eqsteadyns}
			with Navier boundary conditions \eqref{bounNavboual}.
			Multiplying the momentum equation in \eqref{eqsteadyns} by $\varphi_{R}(r)\Bu$, one obtains
			\begin{equation}\label{sec3eqA25}
				\int_{\Omega}-\Delta \Bu\cdot\varphi_{R}\Bu\,d\Bx +\int_{\Omega}(\Bu\cdot\nabla)\Bu\cdot \varphi_{R}\Bu\,d\Bx +\int_{\Omega}\nabla P\cdot\varphi_{R}\Bu\,d\Bx=0.
			\end{equation}	
			For  the first term of \eqref{sec3eqA25}, integrating by parts yields
			\begin{equation}\label{sec511eqidenti1}
				\int_{\Omega}-\Delta\Bu \cdot \varphi_{R}\Bu\,d\Bx
				=\int_{\Omega}|\nabla\Bu|^{2}\varphi_{R}+\nabla \varphi_{R}\cdot \nabla   \Bu\cdot\Bu\,d\Bx -\int_{\partial\Omega}\varphi_{R}\Bn \cdot \nabla\Bu\cdot\Bu\,dS.
			\end{equation}
			Note that on the boundary,
			\begin{equation}
				2\Bn\cdot\D(\Bu)\cdot\Bu-\Bn \cdot \nabla \Bu \cdot \Bu = \sum_{i,\, j=1}^3 n_j \partial_{x_i} u^j u^i.
			\end{equation}
			The impermeable boundary $\Bu\cdot\Bn=0$ implies that  $\nabla_{\Bt}(\Bu\cdot\Bn)=0$. Hence, it holds that
			\begin{equation}\label{sec511flatsid2}
				\sum_{i,\, j=1}^{3}n_{j} \partial_{x_i}u^{j}u^{i}
				=\Bu \cdot \nabla_{\Bt} (\Bu \cdot \Bn)-(\Bu\cdot \nabla_{\Bt} \Bn) \cdot\Bu
				=-(\Bu\cdot \nabla_{\Bt} \Bn) \cdot\Bu, \qquad \text{on} \  \partial \Omega.
			\end{equation}
			Combining \eqref{sec511eqidenti1}--\eqref{sec511flatsid2} with Navier boundary conditions \eqref{bounNavboual} yields
			\begin{equation}\label{sec511eqidentire}
				\begin{split}
					&\int_{\Omega}-\Delta\Bu \cdot \varphi_{R}\Bu  \,d\Bx\\
					=&\int_{\Omega}|\nabla\Bu|^{2}\varphi_{R}+\nabla \varphi_{R}\cdot \nabla  \Bu\cdot\Bu  \,d\Bx-\int_{\partial\Omega}(\Bu\cdot \nabla_{\Bt} \Bn) \cdot   \Bu\varphi_{R}\,dS +2\alpha \int_{\partial\Omega}|\Bu|^{2}\varphi_{R} \,dS.
				\end{split}
			\end{equation}
			On the other hand,
			\[
			\begin{split}
				&\int_{\Omega}-\Delta\Bu \cdot \varphi_{R}\Bu  \,d\Bx=-2\int_{\Omega}{\rm div}\D(\Bu) \cdot \varphi_{R}\Bu \,d\Bx\\
				=&
				2\int_{\Omega}|\D(\Bu)|^{2} \varphi_{R}  \,d\Bx+	2\int_{\Omega}\nabla\varphi_{R}\cdot\D(\Bu)\cdot \Bu  \,d\Bx+2\alpha\int_{\partial\Omega}|\Bu|^{2}\varphi_{R}\,dS.
			\end{split}
			\]
			This, together with \eqref{sec511eqidentire} gives
			\[
			\begin{split}
				&\int_{\Omega}|\nabla\Bu|^{2}\varphi_{R}\,d\Bx\\
				=&
				2\int_{\Omega}|\D(\Bu)|^{2} \varphi_{R}  +\nabla\varphi_{R}\cdot\D(\Bu)\cdot \Bu\,d\Bx-\int_{\Omega}\nabla \varphi_{R}\cdot \nabla\Bu\cdot\Bu  \,d\Bx+\int_{\partial\Omega}(\Bu\cdot \nabla_{\Bt} \Bn) \cdot   \Bu\varphi_{R}\,dS\\
				\leq&
				2\int_{\Omega}|\D(\Bu)|^{2} \varphi_{R}  \,d\Bx+3\int_{\Omega}|\nabla \varphi_{R}||\nabla\Bu||\Bu|\,d\Bx+\|\nabla_{\Bt} \Bn\|_{L^{\infty}(\partial\Omega)}\int_{\partial\Omega}|\Bu|^{2}\varphi_{R}\,dS\\
				\leq & C \left[
				2\int_{\Omega}|\D(\Bu)|^{2} \varphi_{R}  \,d\Bx+\int_{\Omega}|\nabla \varphi_{R}||\nabla\Bu||\Bu|\,d\Bx+2\alpha\int_{\partial\Omega}|\Bu|^{2}\varphi_{R} \,dS \right],
			\end{split}
			\]
			where the boundedness of $\|\nabla_{\Bt} \Bn\|_{L^{\infty}(\partial\Omega)}$ has been used.
			In conclusion, one obtains
			\begin{equation}\label{eqconeqBu}
				\begin{split}
					&\int_{\Omega}|\nabla\Bu|^{2}\varphi_{R}\,d\Bx+2\alpha\int_{\partial\Omega} \left|\Bu\right|^2\varphi_{R}\,dS \\
					\leq& C\left[	\left|\int_{\Omega}-\Delta\Bu \cdot \varphi_{R}\Bu  \,d\Bx\right|+\int_{\Omega}|\nabla\varphi_{R}||\nabla\Bu||\Bu|\,d\Bx \right]\\
					\leq& C\left[\left|\int_{\Omega}( \Bu\cdot\nabla)\Bu\cdot
					\varphi_{R}\Bu\,d\Bx\right| +\left|\int_{\Omega}\nabla P\cdot\varphi_{R}\Bu \,d\Bx\right|+\int_{\Omega}|\nabla\varphi_{R}|
					|\nabla\Bu||\Bu|\,d\Bx\right]\\
					\leq&C \left[\int_{\Omega}|\nabla\varphi_{R}| |\nabla\Bu||\Bu| \,d\Bx+ \left|\int_{\Omega}|\Bu|^{2}\Bu\cdot \nabla\varphi_{R} \,d\Bx\right| +  \left| \int_\Omega P \Bu \cdot \nabla \varphi_R \,d\Bx\right| \right].
				\end{split}
			\end{equation}
			Note that
			\begin{equation}\label{sec32eqA29}
				\int_{\Omega}P\Bu \cdot \nabla \varphi_{R} \,d\Bx
				=-2\pi \int_{R-1}^{R}
				\int_{f_{1}(r)}^{f_{2}(r)}Pu^{r}r\,
				dzdr.
			\end{equation} 
			The divergence free condition for the axisymmetric solution  is
			\[
			\partial_{r}(ru^{r})+\partial_{z}(ru^{z})=0.
			\]
			Note that for  the axisymmetric layer $\Omega=\left\{
			(x_{1}, x_{2}, x_{3}): (x_{1}, x_{2})\in \mathbb{R}^{2},\,\, f_{1}(r)<x_{3}<f_{2}(r) \right\}$, in cylindrical coordinates, the impermeable conditions 	$\Bu\cdot\Bn=0$ tells that
			\begin{equation}\label{eximform1}
				u^{r}(r,f_{i}(r))f^{\prime}_{i}(r)-u^{z}(r,f_{i}(r))=0, \quad i=1,2.
			\end{equation}
			Hence for every fixed $r\geq0$, one has
			\[
			\begin{split}
				\partial_{r}\int_{f_{1}(r)}^{f_{2}(r)}ru^{r}\,dz
				&=\int_{f_{1}(r)}^{f_{2}(r)}\partial_{r}(r u^{r})\, dz+ru^{r}(r, f_{2}(r))f^{\prime}_{2}(r)-ru^{r}(r, f_{1}(r))f^{\prime}_{1}(r)\\
				&=-\int_{f_{1}(r)}^{f_{2}(r)}\partial_{z}(r u^{z})\, dz+ru^{r}(r, f_{2}(r))f^{\prime}_{2}(r)-ru^{r}(r, f_{1}(r))f^{\prime}_{1}(r)
				=0,
			\end{split}
			\]
			where the last equality is due to \eqref{eximform1}.
			This implies that
			\begin{equation}\label{sec32eqA31}
				\int_{f_{1}(r)}^{f_{2}(r)}ru^{r} \,dz
				=\int_{R-1}^{R}\int_{f_{1}(r)}^{f_{2}(r)} ru^r \,dzdr=0.
			\end{equation}
			Owing to \eqref{sec32eqA31} and Lemma \ref{Bogovskii}, 
			there exists a vector valued function $\Ps_R \in H_0^1(D_R; \mathbb{R}^2)$ satisfying \eqref{eqA86}--\eqref{eqA87}.
			
			\emph{Step 2.} \emph {Proof for Case (b) of Theorem \ref{th:02} ($\beta=0$).}
			We are ready to estimate  the right hand side of \eqref{eqconeqBu}.
			For the  first two terms, using H\"older inequality and 	
			Lemma \ref{1Nabineqnoslipes}, one obtains
			\begin{equation}\label{rese5basces1}
				\begin{split}
					\int_{\Omega}|\nabla\varphi_{R}|
					|\nabla\Bu||\Bu| \,d\Bx
					\leq& CR\|\nabla\Bu\|_{L^{2}(D_{R})}\cdot \|\Bu\|_{L^{2}(D_{R})}\\
					\leq& CR\|\nabla\Bu\|_{L^{2}(D_{R})}\left(\|\nabla\Bu\|_{L^{2}(D_{R})}+\|\Bu\|_{L^{2}(\partial D_{R}\cap \partial \Omega)}\right)\\
					\leq& C\|\nabla\Bu\|_{L^{2}(\OR)}\left(\|\nabla\Bu\|_{L^{2}(\OR)}+\|\Bu\|_{L^{2}(\partial\OR\cap \partial \Omega)}\right)
				\end{split}
			\end{equation}
			and
			\begin{equation}\label{res5basces2}
				\begin{split}
					\left|\int_{\Omega}\left|\Bu \right|^2\Bu\cdot \nabla \varphi_{R} \,d\Bx\right|\leq& CR\|\Bu\|^{2}_{L^{4}(D_{R})}\cdot\|u^{r}\|_{L^{2}(D_{R})}\\
					\leq&
					CR\left(\|\nabla\Bu\|_{L^{2}(D_{R})}+\|\Bu\|_{L^{2}(\partial D_{R}\cap\partial \Omega)}\right)^{3}\\
					\leq& CR^{-\frac{1}{2}}\left(\|\nabla\Bu\|_{L^{2}(\OR)}+\|\Bu\|_{L^{2}(\partial\OR\cap \partial \Omega)}\right)^{3}.
				\end{split}
			\end{equation}
			As for the right hand of \eqref{sec32eqA29}, by \eqref{eqA87}, \eqref{eqparrP}--\eqref{eqparpazP}
			and 	
			Lemma \ref{1Nabineqnoslipes}, one has
			\begin{equation}\label{sec51eqA41}
				\begin{split}
					\left| \int_{R-1}^{R}\int_{f_{1}(r)}^{f_{2}(r)}(\partial_{r}u^{r}\partial_{r}\Psi_{R}^{r}+\partial_{z}u^{r}\partial_{z}\Psi_{R}^{r})\,dzdr\right|
					\leq&
					C\|(\partial_{r}, \partial_{z})u^{r}\|_{L^{2}(D_{R})}\cdot\|(\partial_{r}, \partial_{z})\Psi_{R}^{r}\|_{L^{2}(D_{R})} \\
					\leq& CR^{-\frac{1}{2}}\|\nabla\Bu\|_{L^{2}(\OR)}\cdot	R\|u^{r}\|_{L^{2}(D_{R})}\\
					\leq&
					C\left(\|\nabla\Bu\|_{L^{2}(\OR)}+\|\Bu\|_{L^{2}(\partial\OR\cap \partial \Omega)}\right)^{2}
				\end{split}
			\end{equation}
			and
			\begin{equation}\label{sec52eqA42}
				\begin{split}
					&\left|\int_{R-1}^{R}\int_{f_{1}(r)}^{f_{2}(r)}\left[\left(\dfrac{1}{r}\partial_{r}-\dfrac{1}{r^2}\right)u^{r}\right]\Psi_{R}^{r} \, dzdr\right|\\
					\leq&	
					CR^{-1} \cdot R^{-\frac{1}{2}}\left(\|\nabla\Bu\|_{L^{2}(\OR)}+\|\Bu\|_{L^{2}(\partial\OR\cap \partial \Omega)}\right)
					\cdot R\|u^{r}\|_{L^{2}(D_{R})}\\
					\leq&
					CR^{-1}\left(\|\nabla\Bu\|_{L^{2}(\OR)}+\|\Bu\|_{L^{2}(\partial\OR\cap \partial \Omega)}\right)^{2}.
				\end{split}
			\end{equation}
			Furthermore, it holds that
			\begin{equation}\label{sec53eqA43}
				\begin{split}
					&\left|\int_{R-1}^{R}\int_{f_{1}(r)}^{f_{2}(r)}\left[(u^{r}\partial_{r}+u^{z}\partial_{z})u^{r}\right]\Psi_{R}^{r} \, dzdr\right|\\	
					\leq&
					C\|\Bu\|_{L^{4}(D_{R})}\|(\partial_{r}, \partial_{z})u^{r}\|_{L^{2}(D_{R})}\cdot \|\Psi_{R}^{r}\|_{L^{4}(D_{R})}\\
					\leq&
					CR^{-\frac{1}{2}}\left(\|\nabla\Bu\|_{L^{2}(\OR)}+\|\Bu\|_{L^{2}(\partial\OR\cap \partial \Omega)}\right)\cdot R^{-\frac{1}{2}}\|\nabla \Bu\|_{L^{2}(\OR)}\cdot\|(\partial_{r},\partial_{z})\Psi_{R}^{r}\|_{L^{2}(D_{R})}\\
					\leq&
					CR^{-1}\left(\|\nabla\Bu\|_{L^{2}(\OR)}+\|\Bu\|_{L^{2}(\partial\OR\cap \partial \Omega)}\right)\cdot \|\nabla \Bu\|_{L^{2}(\OR)}\cdot R\|u^{r}\|_{L^{2}(D_{R})}\\
					\leq&
					CR^{-\frac{1}{2}}\left(\|\nabla\Bu\|_{L^{2}(\OR)}+\|\Bu\|_{L^{2}(\partial\OR\cap \partial \Omega)}\right)^{3}
				\end{split}
			\end{equation}
			and
			\begin{equation}\label{sec53eqA443}
				\begin{split}
					&\left|\int_{R-1}^{R}\int_{f_{1}(r)}^{f_{2}(r)}\left[\dfrac{(u^{\theta})^2}{r}\right]\Psi_{R}^{r} \, dzdr\right|\\
					\leq&
					CR^{-1}\|\Bu\|^{2}_{L^{4}(D_{R})} \cdot \|\Psi_{R}^{r}\|_{L^{2}(D_{R})}\\
					\leq&
					CR^{-1}\cdot R^{-1}\left(\|\nabla\Bu\|_{L^{2}(\OR)}+\|\Bu\|_{L^{2}(\partial\OR\cap \partial \Omega)}\right)^{2} \cdot R\|u^{r}\|_{L^{2}(D_{R})}\\
					\leq&
					CR^{-\frac{3}{2}}\left(\|\nabla\Bu\|_{L^{2}(\OR)}+\|\Bu\|_{L^{2}(\partial\OR\cap \partial \Omega)}\right)^{3}.
				\end{split}
			\end{equation}
			Collecting the estimates \eqref{sec51eqA41}--\eqref{sec53eqA443} yields
			\begin{equation}\label{eqse5A441}
				\begin{split}
					&\left|\int_{R-1}^{R}\int_{f_{1}(r)}^{f_{2}(r)}\partial_{r}P\Psi^{r}_{R}\,dzdr\right|\\
					\leq& C\left(\|\nabla\Bu\|_{L^{2}(\OR)}+\|\Bu\|_{L^{2}(\partial\OR\cap\partial\Omega)}\right)^{2}+CR^{-\frac{1}{2}}\left(\|\nabla\Bu\|_{L^{2}(\OR)}+\|\Bu\|_{L^{2}(\partial\OR\cap \partial \Omega)}\right)^{3}.
				\end{split}
			\end{equation}
			Similarly, one has
			\begin{equation}\label{eqsec5A452}
				\begin{split}
					&\left|\int_{R-1}^{R}\int_{f_{1}(r)}^{f_{2}(r)}\partial_{z}P\Psi^{z}_{R}\,dzdr\right|\\
					\leq& C\left(\|\nabla\Bu\|_{L^{2}(\OR)}+\|\Bu\|_{L^{2}(\partial\OR\cap\partial\Omega)}\right)^{2}+CR^{-\frac{1}{2}}\left(\|\nabla\Bu\|_{L^{2}(\OR)}+\|\Bu\|_{L^{2}(\partial\OR\cap \partial \Omega)}\right)^{3}.
				\end{split}
			\end{equation}
			Combining \eqref{rese5basces1}--\eqref{res5basces2} and \eqref{eqse5A441}--\eqref{eqsec5A452}, one arrives at
			\begin{equation}\label{sec512eqBNIii34}
				\begin{split}
					&\int_{\Omega}|\nabla\Bu|^{2}\varphi_{R}\,d\Bx+2\alpha\int_{\partial\Omega} \left|\Bu\right|^2\varphi_{R}\, dS\\
					\leq& C\left(\|\nabla\Bu\|_{L^{2}(\OR)}+\|\Bu\|_{L^{2}(\partial\OR\cap\partial\Omega)}\right)^{2}+CR^{-\frac{1}{2}}\left(\|\nabla\Bu\|_{L^{2}(\OR)}+\|\Bu\|_{L^{2}(\partial\OR\cap \partial \Omega)}\right)^{3}.
				\end{split}
			\end{equation}
			
			Let
			\begin{equation}\label{eqZRinsec5}
				\begin{split}
					Z(R)
					=\int_{\Omega}|\nabla \Bu|^{2}\varphi_{R}\, d\Bx+2\alpha\int_{\partial\Omega}|\Bu|^{2}
					\varphi_{R}\, dS.
				\end{split}
			\end{equation}
			Direct computations give
			\[
			Z^{\prime}(R)=\int_{\OR}|\nabla\Bu|^{2}\,d\Bx+2\alpha\int_{\partial\OR\cap \partial \Omega}|\Bu|^{2}\,dS.
			\]
			Hence the estimate \eqref{sec512eqBNIii34} can be written as
			\[
			Z(R)\leq CZ^{\prime}(R)+CR^{-\frac{1}{2}}\left[Z^{\prime}(R)\right]^{\frac{3}{2}}.
			\]
			As  the proof in Subsection \ref{Sec31}, this implies $Z(R)\equiv 0$ and  $\Bu$ is a constant.  Furthermore, the axisymmetry of $\Bu$ yields $u^{r}=u^{\theta}\equiv0$,  together with impermeable conditions \eqref{eximform1} implies that $u^{z}\equiv0$.  Hence the proof  for Case (b) of Theorem \ref{th:02} ($\beta=0$)  is completed.
			
			\emph{Step 3.} \emph {Brief proof for Case (a) of Theorem \ref{th:02} ($\beta=0$).} It is similar with the proof  for  Case (a) of Theorem \ref{th:01} ($\beta=0$). Note that \eqref{sec32eqA31} gives the Poincar\'e for $u^{r}$, i.e.,
			\[
			\|u^{r}\|_{L^{2}(\OR)}\leq C\|\partial_{z} u^{r}\|_{L^{2}(\OR)}.
			\]
			So it holds that \eqref{eqA104}--\eqref{eqA110}. Collecting the estimates \eqref{eqconeqBu}, \eqref{eqA104}--\eqref{eqA110} gives
			\[
			Z(R)\leq CR^{\frac{1}{2}}\left(\|\Bu\|_{L^{\infty}(\OR)}+\|\Bu\|^{2}_{L^{\infty}(\OR)}\right) \left[Z^{\prime}(R)\right]^{\frac{1}{2}},
			\]
			where $Z(R)$  is defined in \eqref{eqZRinsec5}.
			According to the proof of Theorem \ref{th:01} with $\beta=0$ and the Case (b) of Theorem \ref{th:02} with $\beta=0$, one has $\Bu\equiv0$.
			This completes the proof of Theorem \ref{th:02} with $\beta=0$.
		\end{proof}
		\section{Flows in a growing layer with unbounded outlets}\label{Sec4}	
		In this section, we consider the  axisymmetric flow in a growing layer with unbounded outlets.
		The new ingredients in the proof are the choice of cut-off function   and truncated domains, which are defined in \eqref{cut-off} and \eqref{truncatdomOR} in Section \ref{Sec3}.
		First, we denote
		\[
		\gamma^{\ast}=(4\gamma)^{-1},
		\]
		where $\gamma$ is given in \eqref{eqdomainrate}.
		Clearly, one has
		\begin{equation*}
			\|f^{\prime}\|_{L^{\infty}(\partial\Omega)}=2\gamma =(2\gamma^{\ast})^{-1}
		\end{equation*}
		and
		\begin{equation}\label{domainc1}
			\frac{1}{2}f(r)\leq f(\xi)\leq \frac{3}{2}f(r) \quad \text{for any}
			\,\, \xi\in [r-\gamma^{\ast}f(r),\, r+\gamma^{\ast}f(r)].
		\end{equation}
		Note that $f(R)\leq O(R^{\beta})$, $0\leq \beta<\dfrac{4}{5}$. For $R$ large enough, it holds that $R> \gamma^{\ast} f(R)$.
		Then we define a new  cut-off function
		\begin{equation} \label{newcut-off}
			\hat{\varphi}_R(r) = \left\{ \ba
			&\gamma^{\ast},\ \ \ \ \ \ \ \ \ \ r < R-\gamma^{\ast}f(R), \\
			&\frac{R-r}{f(R)},\ \ \ \ R-\gamma^{\ast}f(R) \leq r \leq R, \\
			&0, \ \ \ \ \ \ \ \ \ \ r >  R.
			\ea  \right.
		\end{equation}
		Clearly, it holds
		\begin{equation}\label{fRderivative1}
			\left|\frac{\partial \hat{\varphi}_R(r) }{\partial r} \right|
			=\frac{1}{f(R)}, \quad r\in (R-\gamma^{\ast}f(R), \, R)
		\end{equation}
		and
		\begin{equation}\label{eqvarbigf23}
			\frac{\partial \hat{\varphi}_R(r)}{\partial R} =\frac{1}{f(R)}\left[1-\frac{R-r}{f(R)}f^{\prime}(R)\right]\geq
			\frac{1}{2}[f(R)]^{-1}.
		\end{equation}
		The new truncated domains are defined as follows,  
		\[
		\widetilde{D_{R}}
		=\left\{(r,z): R-\gamma^{\ast}f(R)<r< R,\,\, f_{1}(r)<z<f_{2}(r)\right\}
		\]
		and
		\[\qquad
		\widetilde{\mathscr{O}_{R}} =
		\left\{
		\Bx\in\mathbb{R}^{3}:
		(x_{1}, x_{2})\in B_{R} \setminus \overline{B_{R-\gamma^{\ast}f(R)}},\,\,
		f_{1}(r)<x_{3}< f_{2}(r) \right\},
		\]
		where
		$r=\sqrt{x_{1}^{2}+x_{2}^{2}}$ and
		$B_{R}=\left\{(x_{1}, x_{2}):x_1^2+x_2^2 <R^{2} \right\}$.
		\subsection{No-slip boundary conditions case}
		\label{sec41}
	\begin{proof}[Proof   of Theorem \ref{th:01}] 	The proof contains three steps.
		
		\emph{Step 1.} {\emph{Set up.}} As proved in the Subsection \ref{Sec31}, multiplying the first equation in \eqref{eqsteadyns}  by $\hat{\varphi}_{R}(r)\Bu$ and integrating by parts, one obtains
		\begin{equation}\label{mueqine22new}
			\int_{\Omega}|\nabla\Bu|^{2}\hat{\varphi}_{R}\,d\Bx=-\int_{\Omega}\nabla\hat{\varphi}_{R}\cdot\nabla\Bu\cdot \Bu\,d\Bx+\int_{\Omega}\dfrac{1}{2}|\Bu|^{2}\Bu\cdot\nabla\hat{\varphi}_{R}\,d\Bx+\int_{\Omega}P\Bu\cdot \nabla \hat{\varphi}_{R}\, d\Bx.
		\end{equation}
		By \eqref{fRderivative1}, one has
		\begin{equation}\label{basic1sec4}
			\begin{split}
				\left|\int_{\Omega}\nabla\hat{\varphi}_{R}\cdot \nabla\Bu\cdot\Bu \,d\Bx \right|
				&=\left|\int_{0}^{2\pi}\int_{R-\gamma^{\ast} f(R)}^{R}\int_{f_{1}(r)}^{f_{2}(r)}\nabla\hat{\varphi}_{R}\cdot \nabla\Bu\cdot\Bu r\, dzdrd\theta\right|\\
				&\leq
				CR\cdot [f(R)]^{-1}\|\nabla\Bu\|_{L^{2}(\wtDR)}\|\Bu\|_{L^{2}(\wtDR)}\\
				&\leq
				C[f(R)]^{-1}\|\nabla\Bu\|_{L^{2}(\wtOR)}\cdot M_{1}(\wtDR)\|\nabla\Bu\|_{L^{2}(\wtOR)}\\
				&\leq
				C\|\nabla\Bu\|^{2}_{L^{2}(\wtOR)},  \end{split}
		\end{equation}
		where the third line is due to Lemma \ref{L4ineqnoslipes} and $M_1(\wtDR)$ satisfies 
		\[
		C^{-1}f(R)\leq M_{1}(\wtDR)\leq Cf(R).
		\]
		Similarly,
		\begin{equation}\label{basic2sec4}
			\begin{split}
				\left|\int_{\Omega}|\Bu|^{2}\Bu\cdot \nabla\hat{\varphi}_{R}\, d\Bx\right|
				=&
				[f(R)]^{-1}\left|\int_{0}^{2\pi} \int_{R-\gamma^{\ast} f(R)}^{R}\int_{f_{1}(r)}^{f_{2}(r)}|\Bu|^{2}u^{r} r\, dzdrd\theta\right|\\
				\leq&
				C R\cdot [f(R)]^{-1}\cdot \|\Bu\|^{2}_{L^{4}(\wtDR)}\|u^{r}\|_{L^{2}(\wtDR)}\\
				\leq&
				CR\cdot [f(R)]^{-1}\cdot M^{2}_{2}(\wtDR)M_{1}(\wtDR)\|\nabla\Bu\|^{3}_{L^{2}(\wtDR)}\\
				\leq&
				CR^{-\frac{1}{2}}f(R)
				\|\nabla\Bu\|^{3}_{L^{2}(\wtOR)},
			\end{split}
		\end{equation}
		where the third line is due to Lemma \ref{L4ineqnoslipes}, and here $M_{2}(\wtDR)$ satisfies
		\[
		C^{-1} [f(R)]^{\frac{1}{2}}\leq M_{2}(\wtDR)\leq
		C[f(R)]^{\frac{1}{2}}.
		\]
		Herein,
		\begin{equation}\label{eqP12re}
			\int_{\Omega}P\Bu\cdot \nabla \hat{\varphi}_{R}\,d\Bx=-\frac{2\pi}{f(R)}\int_{R-\gamma^{\ast} f(R)}^{R}\int_{f_{1}(r)}^{f_{2}(r)}Pu^{r}r\,dzdr.
		\end{equation}
		As proved in Subsection  \ref{Sec31}, it holds that
		\[
		\int_{f_{1}(r)}^{f_{2}(r)}ru^{r}\,dz=0 \quad \text{and}\quad \int_{R-\gamma^{\ast}f(R)}^{R}\int_{f_{1}(r)}^{f_{2}(r)}ru^{r}\,dzdr=0.	\]
		By virtue of  Lemma \ref{Bogovskii},
		there exists a vector valued function $\tilde{\Ps}_{R}(r,z)\in H^{1}_{0}(\wtDR;\, \mathbb{R}^{2})$ satisfying
		\begin{equation}\label{Bog1were}
			\partial_{r}\tilde{\Psi}^{r}_{R}+\partial_{z}\tilde{\Psi}^{z}_{R}=ru^{r}
		\end{equation}
		and
		\begin{equation}\label{eqbogindm2}
			\|\partial_{r}\tilde{\Ps}_{R}\|_{L^{2}(\wtDR)}+	\|\partial_{z}\tilde{\Ps}_{R}\|_{L^{2}(\wtDR)}
			\leq C\|ru^{r}\|_{L^{2}(\wtDR)}
			\leq CR^{\frac{1}{2}}\|u^{r}\|_{L^{2}(\wtOR)},
		\end{equation}
		where the constant $C$ is independent $R$, since the domain $\wtDR$ satisfies  \eqref{domainc1}
		and the norm for the Bogovskii  map is bounded by \eqref{eqBoges1}--\eqref{eqBoges2es}.  It can be referred  to \cite{SWX} for a detailed discussion.
		Therefore, combining  \eqref{eqP12re} and \eqref{Bog1were}, one derives
		\begin{equation}\label{eqPressesunbf}
			\begin{split}
				\int_{\Omega}P\Bu\cdot \nabla \hat{\varphi}_{R}\,d\Bx
				=&-\frac{2\pi}{f(R)}\int_{R-\gamma^{\ast} f(R)}^{R}\int_{f_{1}(r)}^{f_{2}(r)}
				P(\partial_{r}\tilde{\Psi}^{r}_{R}+\partial_{z}\tilde{\Psi}^{z}_{R})\,dzdr\\
				=&\, \frac{2\pi}{f(R)}\int_{R-\gamma^{\ast}f(R)}^{R}\int_{f_{1}(r)}^{f_{2}(r)}(\partial_{r}P\tilde{\Psi}^{r}_{R}+\partial_{z}P\tilde{\Psi}^{z}_{R})\,dzdr.
			\end{split}
		\end{equation}
		According to \eqref{eqparrP} and \eqref{eqparpazP}, one has
		\begin{equation}\label{eqrPqw1257}
			\begin{split}
				&\int_{R-\gamma^{\ast} f(R)}^{R}\int_{f_{1}(r)}^{f_{2}(r)}\partial_{r}P\tilde{\Psi}^{r}_{R}\,dzdr\\
				=&
				-\int_{R-\gamma^{\ast}f(R)}^{R}\int_{f_{1}(r)}^{f_{2}(r)}(\partial_{r}u^{r}\partial_{r}\tilde{\Psi}^{r}_{R}+\partial_{z}u^{r}\partial_{z}\tilde{\Psi}^{r}_{R})-\left[\left(\dfrac{1}{r}\partial_{r}-\dfrac{1}{r^2}\right)u^{r}\right]\tilde{\Psi}^{r}_{R}\,dzdr\\
				&-\int_{R-\gamma^{\ast}f(R)}^{R}\int_{f_{1}(r)}^{f_{2}(r)}\left[(u^{r}\partial_{r}+u^{z}\partial_{z})u^{r}-\dfrac{(u^{\theta})^2}{r}\right]\tilde{\Psi}^{r}_{R}\,dzdr
			\end{split}
		\end{equation}
		and
		\begin{equation}\label{eqrPqw234}
			\begin{split}
				&\int_{R-\gamma^{\ast} f(R)}^{R}\int_{f_{1}(r)}^{f_{2}(r)}\partial_{z}P\tilde{\Psi}^{z}_{R}\,dzdr\\
				=&
				-\int_{R-\gamma^{\ast}f(R)}^{R}\int_{f_{1}(r)}^{f_{2}(r)}\left[(\partial_{r}u^{z}\partial_{r}\tilde{\Psi}^{z}_{R}+\partial_{z}u^{z}\partial_{z}\tilde{\Psi}^{z}_{R})-\left(\dfrac{1}{r}\partial_{r}u^{z}\tilde{\Psi}^{z}_{R}\right)\right]\,dzdr\\
				&-\int_{R-\gamma^{\ast}f(R)}^{R}\int_{f_{1}(r)}^{f_{2}(r)}\left[(u^{r}\partial_{r}+u^{z}\partial_{z})u^{z}\right]\tilde{\Psi}^{z}_{R}\,dzdr.
			\end{split}
		\end{equation}
		\emph{Step 2.} {\emph{Proof for Case (b) of Theorem \ref{th:01}.}} Now we begin to estimate the right hand of \eqref{eqrPqw1257} and \eqref{eqrPqw234}.
		By \eqref{Bog1were}--\eqref{eqbogindm2} and Lemma \ref{L4ineqnoslipes},
		one obtains
		\begin{equation}\label{eqesPre1}
			\begin{split}
				&\left|\int_{R-\gamma^{\ast} f(R)}^{R}\int_{f_{1}(r)}^{f_{2}(r)}(\partial_{r}u^{r}\partial_{r}\tilde{\Psi}^{r}_{R}+\partial_{z}u^{r}\partial_{z}\tilde{\Psi}^{r}_{R})\,dzdr\right|\\
				\leq&
				C\|(\partial_{r},\partial_{z})u^{r}\|_{L^{2}(\wtDR)}\|(\partial_{r},\partial_{z})\tilde{\Psi}^{r}_{R}\|_{L^{2}(\wtDR)}\\
				\leq&
				CR^{-\frac{1}{2}}\|\nabla \Bu\|_{L^{2}(\wtOR)}\cdot \|ru^{r}\|_{L^{2}(\wtDR)}\\
				\leq&
				CR^{-\frac{1}{2}}\|\nabla \Bu\|_{L^{2}(\wtOR)}\cdot R^{\frac{1}{2}}M_{1}(\wtDR)\|\nabla\Bu\|_{L^{2}(\wtOR)}\\
				\leq&
				Cf(R)\|\nabla \Bu\|^{2}_{L^{2}(\wtOR)}
			\end{split}
		\end{equation}  	
		and
		\begin{equation}\label{eqesPre21n}
			\begin{split}
				&\left|\int_{R-\gamma^{\ast} f(R)}^{R}\int_{f_{1}(r)}^{f_{2}(r)}\left[\left(\frac{1}{r}\partial_{r}-\frac{1}{r^{2}}\right)u^{r}\right]\tilde{\Psi}^{r}_{R}\,dzdr\right|\\
				\leq& C\left(R^{-1}\|\partial_{r}u^{r}\|_{L^{2}(\wtDR)}+R^{-2}\|u^{r}\|_{L^{2}(\wtDR)}\right)\|\tilde{\Psi}^{r}_{R}\|_{L^{2}(\wtDR)}\\
				\leq& C\left(R^{-\frac{3}{2}}+R^{-\frac{5}{2}}M_{1}(\wtDR)\right)\|\nabla\Bu\|_{L^{2}(\wtOR)}\cdot M_{1}(\wtDR)\cdot R^{\frac{1}{2}}\|u^{r}\|_{L^{2}(\wtOR)}\\
				\leq& C\left(R^{-1}[f(R)]^{2}+R^{-2}[f(R)]^{3}\right)\|\nabla\Bu\|^{2}_{L^{2}(\wtOR)}.
			\end{split}
		\end{equation}
		Furthermore, it holds that
		\begin{equation}\label{eqessec41}
			\begin{split}
				&\left|\int_{R-\gamma^{\ast} f(R)}^{R}\int_{f_{1}(r)}^{f_{2}(r)}\left[\left(u^{r}\partial_{r}+u^{z}\partial_{z}\right)u^{r}\right]\tilde{\Psi}^{r}_{R}\,dzdr\right|\\
				\leq&
				C\|\Bu\|_{L^{4}(\wtDR)} \|(\partial_{r},\partial_{z})u^{r}\|_{L^{2}(\wtDR)}\cdot\|\tilde{\Psi}^{r}_{R}\|_{L^{4}(\wtDR)}\\
				\leq& CR^{-\frac{1}{2}}M_{2}(\wtDR)\|\nabla\Bu\|_{L^{2}(\wtOR)}\cdot R^{-\frac{1}{2}}\|\nabla\Bu\|_{L^{2}(\wtOR)}\cdot
				M_{2}(\wtDR)\|\nabla\tilde{\Psi}^{r}_{R}\|_{L^{2}(\wtDR)}\\
				\leq& CR^{-1}M^{2}_{2}\|\nabla\Bu\|^{2}_{L^{2}(\wtOR)}\|ru^{r}\|_{L^{2}(\wtDR)}\\
				\leq& CR^{-1}M^{2}_{2}\|\nabla\Bu\|^{2}_{L^{2}(\wtOR)}\cdot R^{\frac{1}{2}}M_{1}\|\nabla\Bu\|_{L^{2}(\wtOR)}\\
				\leq&
				CR^{-\frac{1}{2}}[f(R)]^{2}\|\nabla\Bu\|^{3}_{L^{2}(\wtOR)}
			\end{split}
		\end{equation}
		and
		\begin{equation}\label{eqesPre41}
			\begin{split}
				&\left|\int_{R-\gamma^{\ast} f(R)}^{R}\int_{f_{1}(r)}^{f_{2}(r)}\left[\frac{(u^{\theta})^{2}}{r}\right]\tilde{\Psi}^{r}_{R}\,dzdr\right|\\
				\leq&
				CR^{-1}\|\Bu\|^{2}_{L^{4}(\wtDR)}\cdot\|\tilde{\Psi}^{r}_{R}\|_{L^{2}(\wtDR)}\\
				\leq&
				CR^{-1}\cdot R^{-1} M^{2}_{2}(\wtDR)\|\nabla\Bu\|^{2}_{L^{2}(\wtOR)}\cdot M_{1}(\wtDR)\|\nabla\tilde{\Psi}^{r}_{R}\|_{L^{2}(\wtDR)}\\
				\leq&
				CR^{-2}M^{2}_{2}M_{1} \|\nabla\Bu\|^{2}_{L^{2}(\wtOR)}\cdot R^{\frac{1}{2}}M_{1}\|\nabla\Bu\|_{L^{2}(\wtOR)}\\
				\leq&
				CR^{-\frac{3}{2}}[f(R)]^{3}\|\nabla\Bu\|^{3}_{L^{2}(\wtOR)}.
			\end{split}
		\end{equation}
		Combining the estimates \eqref{eqesPre1}--\eqref{eqesPre41} and due to $f(R)\leq O(R^{\beta})$ as in \eqref{growboundabe1}, where $0\leq \beta<\dfrac{4}{5}$ is required, one derives
		\[
		\begin{split}
			\left|\int_{R-\gamma^{\ast} f(R)}^{R}\int_{f_{1}(r)}^{f_{2}(r)}\partial_{r}P\tilde{\Psi}^{r}_{R}\,dzdr\right|
			\leq Cf(R)\|\nabla \Bu\|^{2}_{L^{2}(\wtOR)}+CR^{-\frac{1}{2}}[f(R)]^{2}\|\nabla\Bu\|^{3}_{L^{2}(\wtOR)}.
		\end{split}
		\]
		Similarly, it holds that
		\[
		\begin{split}
			\left|\int_{R-\gamma^{\ast} f(R)}^{R}\int_{f_{1}(r)}^{f_{2}(r)}\partial_{z}P\tilde{\Psi}^{z}_{R}\,dzdr\right|
			\leq Cf(R)\|\nabla \Bu\|^{2}_{L^{2}(\wtOR)}+CR^{-\frac{1}{2}}[f(R)]^{2}\|\nabla\Bu\|^{3}_{L^{2}(\wtOR)}.
		\end{split}
		\]
		Thus from \eqref{eqPressesunbf}, one arrives at
		\[
		\left|\int_{\Omega}P\Bu\cdot \nabla \hat{\varphi}_{R}\,d\Bx\right|\leq C\|\nabla \Bu\|^{2}_{L^{2}(\wtOR)}+CR^{-\frac{1}{2}}f(R)\|\nabla\Bu\|^{3}_{L^{2}(\wtOR)}.
		\]
		This, together with \eqref{mueqine22new}, \eqref{basic1sec4}--\eqref{basic2sec4}, yields
		\begin{equation}\label{escon237}
			\int_{\Omega}|\nabla\Bu|^{2}\hat{\varphi}_{R}\,d\Bx
			\leq C\|\nabla\Bu\|^{2}_{L^{2}(\wtOR)}+CR^{-\frac{1}{2}}f(R)\|\nabla\Bu\|^{3}_{L^{2}(\wtOR)}.
		\end{equation}
		
		Let
		\begin{equation}\label{YRdefq1}
			\hat{Y}(R)=
			\int_{\Omega}|\nabla\Bu|^{2}\hat{\varphi}_{R}\,d\Bx.
		\end{equation}
		The explicit form \eqref{newcut-off} of   $\hat{\varphi}_{R}$  gives
		\begin{equation}
			\hat{Y}(R)
			=2\pi\left(\int_{0}^{R-\gamma^{\ast}f(R)}\int_{f_{1}(r)}^{f_{2}(r)}\gamma^{\ast}|\nabla\Bu|^{2}r\,dzdr+\int_{R-\gamma^{\ast}f(R)}^{R}\int_{f_{1}(r)}^{f_{2}(r)}|\nabla\Bu|^{2}\frac{R-r}{f(R)}r \,dzdr\right)
		\end{equation}
		and
		\begin{equation}\label{Yprimedef}
			\begin{split}
				\hat{Y}^{\prime}(R)=&2\pi\int_{R-\gamma^{\ast} f(R))}^{R}\int_{f_{1}(r)}^{f_{2}(r)}
				|\nabla\Bu|^{2}\partial_{R}  \hat{\varphi}_{R}\cdot r\,dzdr\\
				\geq& \frac{1}{2}[f(R)]^{-1}\int_{\wtOR}|\nabla\Bu|^{2}\,d\Bx,
			\end{split}
		\end{equation}
		where the second line is due to \eqref{eqvarbigf23}.
		Hence the estimate \eqref{escon237} can be written as
		\begin{equation}\label{eqesestYhat}
			\begin{split}
				\hat{Y}(R)
				\leq&
				C\left\{f(R)\hat{Y}^{\prime}(R)+R^{-\frac{1}{2}}[f(R)]^{\frac{5}{2}}[\hat{Y}^{\prime}(R)]^{\frac{3}{2}} \right\}\\
				\leq& C\left\{R^{\beta}\hat{Y}^{\prime}(R)+R^{\frac{5\beta-1}{2}}[\hat{Y}^{\prime}(R)]^{\frac{3}{2}} \right\}.
			\end{split}
		\end{equation}
		If $\hat{Y}(R)$ is not identically zero,  it from Case (b) of Lemma \ref{le:differineq},
		\begin{equation}\label{inYRinbet}
			\varliminf_{R \rightarrow + \infty}  R^{5\beta-4}\hat{Y}(R)>0.
		\end{equation}
		However, since \eqref{DassumDnoslip1} one has $\displaystyle\varliminf_{R \rightarrow + \infty}  R^{5\beta-4}\hat{Y}(R)=0$. The contradiction implies $\hat{Y}(R)\equiv0$.
		Hence, $\nabla\Bu\equiv 0$, together with the no-slip boundary conditions \eqref{curboundarynoslip} yields $\Bu\equiv 0$.
		This finishes the proof for Case (b) of Theorem \ref{th:01}.
		
		\emph{Step 3.} {\emph{Proof for Case (a) of Theorem \ref{th:01}.}} Next, we estimate the right hand side of \eqref{mueqine22new} in a different way.
		First, since $\gamma^{\ast} f(R)<R$, it holds that
		\[
		\begin{split}
			\left(\int_{0}^{2\pi}\int_{R-\gamma^{\ast}f(R)}^{R}\int_{f_{1}(r)}^{f_{2}(r)}\Bu^{2} r\, dz dr d\theta\right)^{\frac{1}{2}}
			&\leq 
			\left(\int_{0}^{2\pi}\int_{R-\gamma^{\ast}f(R)}^{R}\int_{f_{1}(r)}^{f_{2}(r)}r\,dz drd\theta\right)^{\frac{1}{2}}\|\Bu\|_{L^{\infty}(\wtOR)}\\
			&\leq CR^{\frac{1}{2}}f(R)  \|\Bu\|_{L^{\infty}(\wtOR)},
		\end{split}
		\]
		which implies
		\[
		\|\Bu\|_{L^{2}(\wtOR)}\leq CR^{\frac{1}{2}}f(R)  \|\Bu\|_{L^{\infty}(\wtOR)}.
		\]
		By H{\"o}lder inequality and Poincar\'e inequality \eqref{Poincare1}, one obtains
		\begin{equation}\label{eqesti1}
			\begin{split}
				\left|\int_{\Omega}\nabla\hat{\varphi}_{R}\cdot \nabla\Bu\cdot\Bu\,d\Bx\right|
				=&\left|\int_{0}^{2\pi}\int_{R-\gamma^{\ast} f(R)}^{R}\int_{f_{1}(r)}^{f_{2}(r)}\nabla\hat{\varphi}_{R}\cdot \nabla\Bu\cdot\Bu r\, dz dr d\theta\right|\\
				\leq& C[f(R)]^{-1}\|\nabla\Bu\|_{L^{2}(\wtOR)}\|\Bu\|_{L^{2}(\wtOR)}\\
				\leq& C[f(R)]^{-1}\|\nabla\Bu\|_{L^{2}(\wtOR)} \cdot R^{\frac{1}{2}}f(R)\|\Bu\|_{L^{\infty}(\wtOR)}\\
				\leq& CR^{\frac{1}{2}}\|\nabla\Bu\|_{L^{2}(\wtOR)} \|\Bu\|_{L^{\infty}(\wtOR)}
			\end{split}
		\end{equation}
		and
		\begin{equation}\label{eqestes2}
			\begin{split}
				\left|\int_{\Omega}|\Bu|^{2}\Bu\cdot \nabla\hat{\varphi}_{R}\, d\Bx\right|
				\leq&
				C[f(R)]^{-1}\cdot R\|\Bu\|^{2}_{L^{\infty}(\wtDR)}\int^{R}_{R-\gamma^{\ast}f(R)}\int_{f_{1}(r)}^{f_{2}(r)}|u^{r}|\,dzdr\\
				\leq&
				C[f(R)]^{-1}\cdot R\|\Bu\|^{2}_{L^{\infty}(\wtDR)}\cdot f(R)\|u^{r}\|_{L^{2}(\wtDR)}
				\\
				\leq&
				CR \|\Bu\|^{2}_{L^{\infty}(\wtOR)}\cdot M_{1}(\wtDR)\|\partial_{z} u^{r}\|_{L^{2}(\wtDR)}\\
				\leq& CR^{\frac{1}{2}}f(R)\|\Bu\|^{2}_{L^{\infty}(\wtOR)}\|\nabla \Bu\|_{L^{2}(\wtOR)}.
			\end{split}
		\end{equation}
		As for the right hand of \eqref{eqrPqw1257}, one has
		\begin{equation}\label{eqes1pre}
			\begin{split}
				&\left|\int_{R-\gamma^{\ast} f(R)}^{R}\int_{f_{1}(r)}^{f_{2}(r)}(\partial_{r}u^{r}\partial_{r}\tilde{\Psi}^{r}_{R}+\partial_{z}u^{r}\partial_{z}\tilde{\Psi}^{r}_{R})\,dzdr\right|\\
				\leq& C\|(\partial_{r},\partial_{z})u^{r}\|_{L^{2}(\wtDR)}\|(\partial_{r},\partial_{z})\tilde{\Psi}^{r}_{R}\|_{L^{2}(\wtDR)}\\
				\leq&
				CR^{-\frac{1}{2}}\|\nabla \Bu\|_{L^{2}(\wtOR)}\cdot R\|u^{r}\|_{L^{2}(\wtDR)}\\
				\leq& CR^{\frac{1}{2}}f(R)\|\nabla \Bu\|_{L^{2}(\wtOR)}\| \Bu\|_{L^{\infty}(\wtOR)}	
			\end{split}
		\end{equation}  	
		and
		\begin{equation}
			\begin{split}
				&\left|\int_{R-\gamma^{\ast} f(R)}^{R}\int_{f_{1}(r)}^{f_{2}(r)}\left[\left(\frac{1}{r}\partial_{r}-\frac{1}{r^{2}}\right)u^{r}\right]\tilde{\Psi}^{r}_{R}\,dzdr\right|\\
				\leq& C\left(R^{-1}\|\partial_{r}u^{r}\|_{L^{2}(\wtDR)}+R^{-2}\|u^{r}\|_{L^{2}(\wtDR)}\right)\|\tilde{\Psi}^{r}_{R}\|_{L^{2}(\wtDR)}\\
				\leq& C\left(R^{-\frac{3}{2}}+R^{-\frac{5}{2}}M_{1}(\wtDR)\right)\|\nabla\Bu\|_{L^{2}(\wtOR)}\cdot R^{\frac{1}{2}}M_{1}(\wtDR)\|u^{r}\|_{L^{2}(\wtOR)}\\
				\leq& C\left(R^{-1}M_{1}+R^{-2}M^{2}_{1}\right)\|\nabla\Bu\|_{L^{2}(\wtOR)}\cdot R^{\frac{1}{2}}f(R)
				\|\Bu\|_{L^{\infty}(\wtOR)}\\
				\leq& C\left(R^{-\frac{1}{2}}[f(R)]^{2}+R^{-\frac{3}{2}}[f(R)]^{3}\right)
				\|\nabla\Bu\|_{L^{2}(\wtOR)}\|\Bu\|_{L^{\infty}(\wtOR)}.
			\end{split}
		\end{equation}
		Furthermore, one obtains
		\begin{equation}\label{eqes3pre}
			\begin{split}
				&\left|\int_{R-\gamma^{\ast} f(R)}^{R}\int_{f_{1}(r)}^{f_{2}(r)}\left[\left(u^{r}\partial_{r}+u^{z}\partial_{z}\right)u^{r}-\frac{(u^{\theta})^{2}}{r}\right]\tilde{\Psi}^{r}_{R}\,dzdr\right|\\
				\leq&
				C\|\Bu\|_{L^{\infty}(\wtDR)}\left(\|(\partial_{r},\partial_{z})u^{r}\|_{L^{2}(\wtDR)}+R^{-1}M_{1}(\wtDR)\|\partial_{z}u^{\theta}\|_{L^{2}(\wtDR)}\right)
				\cdot\|\tilde{\Psi}^{r}_{R}\|_{L^{2}(\wtDR)}\\
				\leq& C\|\Bu\|_{L^{\infty}(\wtDR)}\left(R^{-\frac{1}{2}}+R^{-\frac{3}{2}}M_{1}\right)\|\nabla\Bu\|_{L^{2}(\wtOR)}\cdot M_{1}\|\partial_{z}\tilde{\Psi}^{r}_{R}\|_{L^{2}(\wtDR)}\\
				\leq& C\|\Bu\|_{L^{\infty}(\wtDR)}\left(R^{-\frac{1}{2}}M_{1}+R^{-\frac{3}{2}}M^{2}_{1}\right)\|\nabla\Bu\|_{L^{2}(\wtOR)}\cdot R^{\frac{1}{2}}\|u^{r}\|_{L^{2}(\wtOR)}\\
				\leq&
				C\left(R^{\frac{1}{2}}[f(R)]^{2}+R^{-\frac{1}{2}}[f(R)]^{3}\right)\|\nabla\Bu\|_{L^{2}(\wtOR)}\|\Bu\|^{2}_{L^{\infty}(\wtOR)}.
			\end{split}
		\end{equation}
		Collecting the estimates   \eqref{eqes1pre}--\eqref{eqes3pre} and due to $f(R)\leq O(R^{\beta})$ as in \eqref{growboundabe1}, where $0\leq \beta<\dfrac{1}{2}$.  According to \eqref{eqrPqw1257}, one derives
		\begin{equation}\label{esPrespd1}
			\left|\int_{R-\gamma^{\ast} f(R)}^{R}\int_{f_{1}(r)}^{f_{2}(r)}\partial_{r}P\tilde{\Psi}^{r}_{R}\,dzdr\right|\leq CR^{\frac{1}{2}}\left(f(R)\|\Bu\|_{L^{\infty}(\wtOR)}+[f(R)]^{2}\|\Bu\|^{2}_{L^{\infty}(\wtOR)}\right)\|\nabla\Bu\|_{L^{2}(\wtOR)}.
		\end{equation}
		Similarly, it holds that
		\begin{equation}\label{esPrespd2}
			\left|\int_{R-\gamma^{\ast} f(R)}^{R}\int_{f_{1}(r)}^{f_{2}(r)}\partial_{z}P\tilde{\Psi}^{z}_{R}\,dzdr\right| \leq CR^{\frac{1}{2}}\left(f(R)\|\Bu\|_{L^{\infty}(\wtOR)}+[f(R)]^{2}\|\Bu\|^{2}_{L^{\infty}(\wtOR)}\right)\|\nabla\Bu\|_{L^{2}(\wtOR)}.
		\end{equation}
		Combining the estimates  \eqref{eqesti1}--\eqref{eqestes2} and \eqref{esPrespd1}--\eqref{esPrespd2}, from
		\eqref{mueqine22new} and \eqref{eqPressesunbf}, one arrives at
		\begin{equation}\label{eqnablauuarr}
			\begin{split}
				\int_{\Omega}|\nabla\Bu|^{2}\hat{\varphi}_{R}\,d\Bx
				\leq CR^{\frac{1}{2}}\left(\|\Bu\|_{L^{\infty}(\wtOR)}+f(R)\|\Bu\|^{2}_{L^{\infty}(\wtOR)}\right)\|\nabla\Bu\|_{L^{2}(\wtOR)}.
			\end{split}
		\end{equation}
		
		If
		$\Bu$ satisfies \eqref{eqgrowinguveu1}, for every $0<\epsilon<1$,
		there is a constant $R_{0}(\epsilon)> \epsilon^{\frac{1}{\beta -1}}$, such that
		\[
		\|\Bu\|_{L^{\infty}(\wtOR)}\leq \epsilon R^{1-2\beta} \quad \text{for\ any}\ R\geq R_{0}.
		\]
		Hence the estimate \eqref{eqnablauuarr} can be written as
		\begin{equation}\label{eqA52}
			\begin{split}
				\hat{Y}(R)\leq& CR^{\frac{1}{2}}\left(\|\Bu\|_{L^{\infty}(\wtOR)}+f(R)\|\Bu\|^{2}_{L^{\infty}(\wtOR)}\right)\|\nabla\Bu\|_{L^{2}(\wtOR)}\\
				\leq& CR^{\frac{1}{2}}\left(\|\Bu\|_{L^{\infty}(\wtOR)}+f(R)\|\Bu\|^{2}_{L^{\infty}(\wtOR)}\right)[f(R)]^{\frac{1}{2}}[\hat{Y}^{\prime}(R)]^{\frac{1}{2}}
				\\
				\leq&
				C\epsilon R^{\frac{5-5\beta}{2}}[\hat{Y}^{\prime}(R)]^{\frac{1}{2}}.
			\end{split}
		\end{equation}
		Here to reach the second line, we  used \eqref{Yprimedef}.
		
		Suppose $\hat{Y}(R)>0$,   according to \eqref{inYRinbet},  $\hat{Y}(R)$ must be unbounded as $R\rightarrow +\infty$. For sufficiently large $R$, integrating \eqref{eqA52} from $R$ to $+\infty$ yields
		\[
		R^{5\beta-4}\hat{Y}(R)\leq (4-5\beta)(C\epsilon)^{2} \quad  \text{for\ any}\ 0\leq\beta<\frac{1}{2}.
		\]
		Since $\epsilon$ can be arbitrarily small, this leads to a contradiction with \eqref{inYRinbet}. Hence,   $\hat{Y}(R)\equiv0$ and so $\nabla \Bu\equiv0$. Then the no-slip boundary conditions  yield  $\Bu\equiv0$. This completes the proof  of Theorem \ref{th:01}.
	\end{proof}
	\subsection{Navier boundary conditions case}
	\label{sec42}
	From the proof that in Subsection \ref{Sec32} and Subsection \ref{sec41},  we can show the Liouville-type theorem for the Navier-Stokes system \eqref{eqsteadyns} in a growing layer supplemented with Navier boundary conditions with unbounded outlets.
	\begin{proof}[Proof  of  Theorem \ref{th:02}] The proof contains two steps.
		
		\emph{Step 1.} {\emph{Proof for Case (b) of Theorem \ref{th:02}.}}
		Following almost the same proof as that in Subsection \ref{Sec32}, one arrives at
		\begin{equation}\label{estmNaunb42}
			\begin{split}
				&\int_{\Omega}|\nabla\Bu|^{2}\hat{\varphi}_{R} \,d\Bx+2\alpha\int_{\partial\Omega} \left|\Bu\right|^2\hat{\varphi}_{R} \,dS\\
				\le & C \left[ \int_{\Omega}|\nabla\hat{\varphi}_{R}| |\nabla\Bu||\Bu| \,d\Bx+\left|\int_{\Omega}|\Bu|^{2}\Bu\cdot \nabla\hat{\varphi}_{R} \,d\Bx\right| +\left| \int_\Omega P \Bu \cdot \nabla \hat{\varphi}_R \,d\Bx\right| \right].
			\end{split}	
		\end{equation}
	For the first two terms of the right hand side of \eqref{estmNaunb42},  by    Lemma \ref{1Nabineqnoslipes}, it holds that
	\begin{equation}\label{newbasic1sec42}
		\begin{split}
			&\int_{\Omega}|\nabla\hat{\varphi}_{R}||\nabla\Bu||\Bu| \,d\Bx\\
			\leq&
			CR\cdot[f(R)]^{-1}\|\nabla\Bu\|_{L^{2}(\wtDR)}\|\Bu\|_{L^{2}(\wtDR)}\\
			\leq&
			C[f(R)]^{-1}\|\nabla\Bu\|_{L^{2}(\wtOR)}\cdot M_{1}(\wtDR)\left(\|\nabla\Bu\|_{L^{2}(\wtOR)}+\|\Bu\|_{L^{2}(\partial\wtOR\cap \partial\Omega)}\right)\\
			\leq&
			C\left(\|\nabla\Bu\|_{L^{2}(\wtOR)}+\|\Bu\|_{L^{2}(\partial\wtOR\cap \partial\Omega)}\right)^{2}
		\end{split}
	\end{equation}
	and
	\begin{equation}\label{newbasic2sec42}
		\begin{split}
			\left|\int_{\Omega}|\Bu|^{2}\Bu\cdot \nabla\hat{\varphi}_{R}\, d\Bx\right|
			=&
			[f(R)]^{-1}\left|\int_{0}^{2\pi}\int_{R-\gamma^{\ast} f(R)}^{R}\int_{f_{1}(r)}^{f_{2}(r)}|\Bu|^{2}u^{r}r\,dzdrd\theta
			\right|\\
			\leq&
			CR\cdot [f(R)]^{-1} \|\Bu\|^{2}_{L^{4}(\wtDR)}\|u^{r}\|_{L^{2}(\wtDR)}\\
			\leq&
			CR^{-\frac{1}{2}} M^{2}_{3}(\wtDR)\left(\|\nabla\Bu\|_{L^{2}(\wtOR)}+\|\Bu\|_{L^{2}(\partial\wtOR \cap \partial\Omega)}\right)^{3}\\
			\leq&
			CR^{-\frac{1}{2}} f(R)\left(\|\nabla\Bu\|_{L^{2}(\wtOR)}+\|\Bu\|_{L^{2}(\partial\wtOR\cap \partial\Omega)}\right)^{3},
		\end{split}
	\end{equation}
	where $M_{1}(\wtDR)$ and $M_3(\wtDR)$ satisfy
	\[
	C^{-1}f(R) \leq M_1(\wtDR) \leq Cf(R), 
	\ \ \ \ \ \ \ C^{-1}[f(R)]^{\frac{1}{2}}\leq M_{3}(\wtDR)\leq
	C[f(R)]^{\frac{1}{2}}.
	\]
	The analysis for the last term of \eqref{estmNaunb42} is similar to that in   Subsection \ref{sec41}. It suffices to estimate \eqref{eqrPqw1257} and \eqref{eqrPqw234}. As in Subsection \ref{Sec32}, \eqref{sec32eqA31} gives the Poincar\'e for $u^{r}$,
	\begin{equation}
		\|u^{r}\|_{L^{2}(\wtDR)}\leq M_1(\wtDR)\|\partial_{z} u^{r}\|_{L^{2}(\wtDR)}.
	\end{equation}
	This implies it holds that \eqref{eqesPre1}--\eqref{eqesPre21n}.
	Furthermore, by \eqref{Bog1were}--\eqref{eqbogindm2}, Lemma \ref{L4ineqnoslipes} and Lemma \ref{1Nabineqnoslipes}, one obtains
	\begin{equation}\label{estisec42234}
		\begin{split}
			&\left|\int_{R-\gamma^{\ast} f(R)}^{R}\int_{f_{1}(r)}^{f_{2}(r)}\left[\left(u^{r}\partial_{r}+u^{z}\partial_{z}\right)u^{r}\right]\tilde{\Psi}^{r}_{R}\,dzdr\right|\\
			\leq&
			C\|\Bu\|_{L^{4}(\wtDR)} \|(\partial_{r},\partial_{z})u^{r}\|_{L^{2}(\wtDR)}\cdot\|\tilde{\Psi}^{r}_{R}\|_{L^{4}(\wtDR)}\\
			\leq& CR^{-\frac{1}{2}}M_{3}(\wtDR)\left(\|\nabla\Bu\|_{L^{2}(\wtOR)}+\|\Bu\|_{L^{2}(\partial\wtOR\cap \partial\Omega)}\right)\cdot R^{-\frac{1}{2}}\|\nabla\Bu\|_{L^{2}(\wtOR)}\cdot
			M_{2}(\wtDR)\|\nabla\tilde{\Psi}^{r}_{R}\|_{L^{2}(\wtDR)}\\
			\leq& CR^{-1}M_{2}M_{3}\left(\|\nabla\Bu\|_{L^{2}(\wtOR)}+\|\Bu\|_{L^{2}(\partial\wtOR\cap \partial\Omega)}\right)^{2}
			\|ru^{r}\|_{L^{2}(\wtDR)}\\
			\leq&
			CR^{-1}M_{2}M_{3}
			\left(\|\nabla\Bu\|_{L^{2}(\wtOR)}+\|\Bu\|_{L^{2}(\partial\wtOR
				\cap\partial\Omega)}\right)^{2}
			\cdot R^{\frac{1}{2}}M_{1}\|\nabla\Bu\|_{L^{2}(\wtOR)}\\
			\leq&
			CR^{-\frac{1}{2}}[f(R)]^{2}\left(\|\nabla\Bu\|_{L^{2}(\wtOR)}+\|\Bu\|_{L^{2}(\partial\wtOR\cap \partial\Omega)}\right)^{3}
		\end{split}
	\end{equation}
	and
	\begin{equation}\label{eqessec42344}
		\begin{split}
			&\left|\int_{R-\gamma^{\ast} f(R)}^{R}\int_{f_{1}(r)}^{f_{2}(r)}\left[\frac{(u^{\theta})^{2}}{r}\right]\tilde{\Psi}^{r}_{R}\,dzdr\right|\\
			\leq&
			CR^{-1}\|\Bu\|^{2}_{L^{4}(\wtDR)}\cdot\|\tilde{\Psi}^{r}_{R}\|_{L^{2}(\wtDR)}\\
			\leq&
			CR^{-1}\cdot R^{-1} M^{2}_{3}(\wtDR)\left(\|\nabla\Bu\|_{L^{2}(\wtOR)}+\|\Bu\|_{L^{2}(\partial\wtOR\cap \partial\Omega)}\right)^{2}
			\cdot M_{1}(\wtDR)\|\nabla\tilde{\Psi}^{r}_{R}\|_{L^{2}(\wtDR)}\\
			\leq&
			CR^{-2}M^{2}_{3}M_{1} \left(\|\nabla\Bu\|_{L^{2}(\wtOR)}+\|\Bu\|_{L^{2}(\partial\wtOR\cap\partial\Omega)}\right)^{2}\cdot
			R^{\frac{1}{2}}M_{1}\|\nabla\Bu\|_{L^{2}(\wtOR)}\\
			\leq&
			CR^{-\frac{3}{2}}[f(R)]^{3}\left(\|\nabla\Bu\|_{L^{2}(\wtOR)}+\|\Bu\|_{L^{2}(\partial\wtOR\cap\partial\Omega)}\right)^{3}.
		\end{split}
	\end{equation}
	Combining \eqref{eqesPre1}--\eqref{eqesPre21n} and \eqref{estisec42234}--\eqref{eqessec42344}, one derives
	\[
	\begin{split}
		&\left|\int_{R-\gamma^{\ast} f(R)}^{R}\int_{f_{1}(r)}^{f_{2}(r)}\partial_{r}P\tilde{\Psi}^{r}_{R}\,dzdr\right|\\
		\leq&
		Cf(R)\|\nabla \Bu\|^{2}_{L^{2}(\wtOR)}
		+CR^{-\frac{1}{2}}[f(R)]^{2}\left(\|\nabla\Bu\|_{L^{2}(\wtOR)}+\|\Bu\|_{L^{2}(\partial\wtOR\cap \partial\Omega)}\right)^{3}.
	\end{split}
	\]
	Similarly, one has
	\[
	\begin{split}
		&\left|\int_{R-\gamma^{\ast} f(R)}^{R}\int_{f_{1}(r)}^{f_{2}(r)}\partial_{z}P\tilde{\Psi}^{z}_{R}\,dzdr\right|\\
		\leq&
		Cf(R)\|\nabla \Bu\|^{2}_{L^{2}(\wtOR)}
		+CR^{-\frac{1}{2}}[f(R)]^{2}\left(\|\nabla\Bu\|_{L^{2}(\wtOR)}+\|\Bu\|_{L^{2}(\partial\wtOR\cap \partial\Omega)}\right)^{3}.
	\end{split}
	\]
	Thus from \eqref{eqPressesunbf}, it holds
	that
	\begin{equation}\label{eqpreesnex42}
		\left|\int_{\Omega}P\Bu\cdot \nabla \hat{\varphi}_{R}\,d\Bx\right|\leq C\|\nabla \Bu\|^{2}_{L^{2}(\wtOR)}+CR^{-\frac{1}{2}}f(R)\left(\|\nabla\Bu\|_{L^{2}(\wtOR)}+\|\Bu\|_{L^{2}(\partial\wtOR\cap \partial\Omega)}\right)^{3}.
	\end{equation}
	Combining \eqref{estmNaunb42}--\eqref{newbasic2sec42} and \eqref{eqpreesnex42}, one arrives at
	\begin{equation}\label{coneqsnabu1}
		\begin{split}
			&\int_{\Omega}|\nabla\Bu|^{2}\hat{\varphi}_{R} \,d\Bx+2\alpha\int_{\partial\Omega} \left|\Bu\right|^2\hat{\varphi}_{R} \,dS \\
			\leq& C\left(\|\nabla\Bu\|_{L^{2}(\wtOR)}
			+\|\Bu\|_{L^{2}(\partial\wtOR\cap \partial\Omega)}\right)^{2}+CR^{-\frac{1}{2}} f(R)\left(\|\nabla\Bu\|_{L^{2}(\wtOR)}+\|\Bu\|_{L^{2}(\partial\wtOR\cap \partial\Omega)}\right)^{3}.
		\end{split}
	\end{equation}
	
	Let
	\begin{equation}\label{eqZRinsec42new}
		\hat{Z}(R)
		=\int_{\Omega}|\nabla\Bu|^{2}\hat{\varphi}_{R}\,d\Bx+2\alpha\int_{\partial\Omega}|\Bu|^{2}\hat{\varphi}_{R}\,dS.
	\end{equation}
	Straightforward computations give
	\[
	\begin{split}
		\hat{Z}^{\prime}(R)
		=&
		\int_{\wtOR}|\nabla\Bu|^{2}\partial_{R}\hat{\varphi}_{R}\,d\Bx+2\alpha\int_{\partial\wtOR\cap\partial\Omega}|\Bu|^{2}\partial_{R}  \hat{\varphi}_{R}\,dS\\
		\geq& \frac{1}{2}[f(R)]^{-1}\left[\int_{\wtOR}|\nabla\Bu|^{2}\,d\Bx+2\alpha\int_{\partial\wtOR\cap\partial\Omega}|\Bu|^{2}\,dS\right].
	\end{split}
	\]
	Hence \eqref{coneqsnabu1} can be written as
	\begin{equation}\label{eqestinessec42}
		\begin{split}
			\hat{Z}(R)\leq&
			C\left\{f(R)\hat{Z}^{\prime}(R)+R^{-\frac{1}{2}}[f(R)]^{\frac{5}{2}}[\hat{Z}^{\prime}(R)]^{\frac{3}{2}}\right\}\\
			\leq&
			C\left\{R^{\beta}\hat{Z}^{\prime}(R)+R^{\frac{5\beta-1}{2}}[\hat{Z}^{\prime}(R)]^{\frac{3}{2}}\right\}.
		\end{split}
	\end{equation}
	Suppose $\hat{Z}(R)$ is not identically zero,  it follows from Lemma \ref{le:differineq} that
	\[
	\varliminf_{R \rightarrow + \infty}  R^{5\beta-4}\hat{Z}(R)>0.
	\]
	An argument similar to the one used in Subsection \ref{sec41} shows that
	$\hat{Z}(R)\equiv0$, which implies $\nabla \Bu|_{\Omega}\equiv0$ and  $\Bu|_{\partial\Omega}\equiv0$. Consequently, $\Bu$ is trivial.
	This completes the proof for Case (b) of Theorem \ref{th:02}.
	
	\emph{Step 2.} {\emph{Brief proof for Case (a) of Theorem \ref{th:02}.}} From \eqref{estmNaunb42}, \eqref{eqPressesunbf} and \eqref{eqesti1}--\eqref{esPrespd2}, it is not diffcult to get
	\[
	\begin{split}
		&\int_{\Omega}|\nabla\Bu|^{2}\hat{\varphi}_{R} \,d\Bx+2\alpha\int_{\partial\Omega}\left|\Bu\right|^2\hat{\varphi}_{R} \, dS\\
		\leq&
		CR^{\frac{1}{2}}\left(\|\Bu\|_{L^{\infty}(\wtOR)}+f(R)\|\Bu\|^{2}_{L^{\infty}(\wtOR)}\right)\cdot\left(\|\nabla\Bu\|_{L^{2}(\wtOR)}+\|\Bu\|_{L^{2}(\partial\wtOR\cap \partial \Omega)}\right).
	\end{split}
	\]
	Following the same argument as that in the  proof  of Theorem \ref{th:01} yields
	$\hat{Z}(R)\equiv0$, then one can show that $\Bu\equiv 0$.  This finishes the proof  of Theorem \ref{th:02}.
\end{proof}	

{\bf Acknowledgement.}
The research was partially supported by the National Key R$\&$D Program of China, Project Number 2020YFA0712000.  The research of Wang was partially supported by NSFC grants 12171349 and 12271389. The research of  Xie was partially supported by  NSFC grants 12250710674 and 11971307, Fundamental Research Grants for Central  universities,  and  Natural Science Foundation of Shanghai 21ZR1433300, Program of Shanghai Academic Research leader 22XD1421400.


\medskip
\begin{thebibliography}{99}
	\bibitem{ASHIJMSJ03} T. Abe and  Y. Shibata,   On a resolvent estimate of the Stokes equation on an infinite layer,
	{\it J. Math. Soc. Japan}, {\bf55} (2003), no. 2,  469--497.
	
	\bibitem{TACGJDE21} P. Acevedo Tapia,  C. Amrouche,   C. Conca, and  A. Ghosh,
	Stokes and Navier-Stokes equations with Navier boundary conditions,
	{\it J. Differ. Equ.}, {\bf285} (2021), 258--320.
	
	\bibitem{ARJDE14} C. Amrouche and A. Rejaiba,
	$L^{p}$-theory for Stokes and Navier-Stokes equations with Navier boundary conditions,
	{\it J. Differ. Equ.}, {\bf256} (2014), no. 4, 1515--1547.
	
	\bibitem{ARARMA11} C. Amrouche and M. \'A. Rodr\'iguez-Bellido,
	Stationary Stokes, Oseen and Navier-Stokes equations with singular data,
	{\it Arch. Ration. Mech. Anal.}, {\bf199} (2011), no. 2, 597--651.	
	
	\bibitem{aBGWX} J. Bang,  C. Gui,  Y. Wang, and C. Xie,
	Liouville-type theorems for steady solutions to the Navier-Stokes system in a slab,
	{\it arXiv: 2205.13259v4}.
	
	
	\bibitem{aBYA} J. Bang and Z. Yang,
	Saint-Venant estimates and Liouville-type theorems for the stationary Navier-Stokes equation in $\mathbb{R}^{3}$,
	{\it arXiv: 2402.11144v1}.
	
	\bibitem{HBV04ADE} H. Beir\~{a}o Da Veiga,
	Regularity for Stokes and generalized Stokes systems under nonhomogeneous slip-type boundary conditions,
	{\it Adv. Differ. Equ.}, {\bf9} (2004), no. 9--10, 1079--1114.
	
	\bibitem{HBV05CpAM} H. Beir\~{a}o Da Veiga,
	On the regularity of flows with Ladyzhenskaya shear-dependent viscosity and slip or nonslip boundary conditions,  {\it Comm. Pure Appl. Math.}, {\bf 58} (2005), no. 4, 552--577.
	
	\bibitem{HBV06Cpaa} H. Beir\~{a}o Da Veiga,
	Vorticity and regularity for flows under the Navier boundary condition,
	{\it Comm. Pure Appl. Anal.}, {\bf5} (2006), no. 4, 907--918.
	
	\bibitem{BDCDS10}L. C. Berselli,
	An elementary approach to the 3D Navier-Stokes equations with Navier boundary conditions: existence and uniqueness of various classes of solutions in the flat boundary case,
	{\it Discrete Contin. Dyn. Syst., Ser. S}, {\bf3} (2010), no. 2, 199--219.
	
	\bibitem{JMFM13bfz} M. Bildhauer, M. Fuchs, and G. Zhang,
	\newblock Liouville-type theorems for steady flows of degenerate power law fluids in the plane,
	\newblock \emph{J. Math. Fluid Mech.}, \textbf{15}  (2013), 583--616.
	
	
	
	\bibitem{B79DANS}M. E. Bogovskii, Solution of the first boundary value problem for an equation of continuity of an incompressible medium, {\it Dokl. Akad. Nauk SSSR}, {\bf 248} (5) (1979), 1037--1040.
	
	%
	
	%
	\bibitem{CPZJFA20}B. Carrillo, X. Pan, and  Q. S. Zhang, Decay and vanishing of some axially symmetric D-solution of the Navier-Stokes equations, {\it J. Funct. Anal.}, {\bf 279} (2020), no. 1, 108504, 49 pp.
	
	\bibitem{CPZZARMA20}B. Carrillo, X.  Pan, Q. S. Zhang, and N. Zhao,  Decay and vanishing of some D-solutions of the Navier-Stokes equations, {\it  Arch. Ration. Mech. Anal.},  {\bf 237} (2020), no. 3, 1383--1419.
	
	%
	\bibitem{CWDCDS16} D. Chae and S. Weng,  Liouville type theorems for the steady axially symmetric Navier-Stokes and Magnetohydrodynamic equations,
	{\it Discrete Contin. Dyn. Syst.}, {\bf36} (2016), 5267--5285.
	%
	%
	%
	%
	%
	%
	%
	\bibitem{CQIndin10} G.-Q. Chen and Z. Qian,
	A study of the Navier-Stokes equations with the kinematic and Navier boundary conditions,  {\it Indiana Univ. Math. J.}, {\bf 59} (2010), no. 2, 721--760.
	
	\bibitem{DLX18Jmfm}S. Ding, Q. Li, and Z. Xin, Stability analysis for the incompressible Navier-Stokes equations with Navier boundary
	conditions, {\it J. Math. Fluid Mech.}, {\bf 20} (2018), no. 2, 603--629.
	
	%
	
	
	\bibitem{GAGP11}G. P. Galdi, An introduction to the mathematical theory of the Navier-Stokes equations. Steady-state problems, Second edition, Springer Monographs in Mathematics. Springer, New York, 2011.
	
	
	
	\bibitem{GWASP78}D. Gilbarg and H. F. Weinberger, Asymptotic properties of steady plane solutions of
	the Navier-Stokes equations with bounded Dirichlet integral, {\it Ann. Scuola Norm. Sup. Pisa Cl. Sci. (4)}, {\bf 5} (1978), no. 2, 381--404.
	
	\bibitem{HWXAHE}
	J. Han, Y. Wang, and C. Xie,
	Liouville-type theorems for steady Navier-Stokes system under helical symmetry  or Navier boundary conditions,
	\emph{arXiv:2312.10382v1}.
	
	
	
	\bibitem{JPKSIAM06}J. P. Kelliher, Navier-Stokes equations with Navier boundary conditions for a bounded domain in the plane, {\it SIAM J. Math. Anal.}, {\bf 38} (2006), no. 1, 210--232.
	
	
	\bibitem{JKKTARMA66}J. K. Knowles, On Saint-Venant's principle in the two-dimensional linear theory of elasticity, {\it Arch. Ration. Mech. Anal.}, {\bf 21} (1966), 1--22.
	
	\bibitem{KNSS09}G. Koch, N. Nadirashvili, G. A. Seregin, and V. Sver\'{a}k, Liouville theorems for the Navier-Stokes equations and applications, {\it Acta Math.}, {\bf 203} (2009), 83--105.
	
	
	
	\bibitem{KPRJMFM15} M. Korobkov, K. Pileckas, and R. Russo,
	The Liouville theorem for the steady-state Navier-Stokes problem for axially symmetric 3D solutions in absence of swirl,
	{\it J. Math. Fluid Mech.}, {\bf17} (2015), 287--293.
	
	\bibitem{KTWJFA17}H. Kozono, Y. Terasawa, and Y. Wakasugi, A remark on Liouville-type theorems for the stationary Navier-Stokes equations in three space dimensions,  {\it J. Funct. Anal.}, {\bf 272} (2017), no. 2, 804--818.
	
	\bibitem{KTWJFA22}
	H. Kozono, Y. Terasawa, and Y. Wakasugi,	Asymptotic properties of steady solutions to the 3D axisymmetric Navier-Stokes equations with no swirl,
	{\it J. Funct. Anal.}, {\bf  282} (2022), no. 2, Paper No. 109289, 21 pp.
	
	\bibitem{KTWarXi23} H. Kozono, Y. Terasawa, and Y. Wakasugi, Liouville-type theorems for the new Taylor-Couette flow of the stationary Navier-Stokes equations,  {\it arXiv: 2305.08451v2}.
	
	
	\bibitem{LSZNSL80} O. A. Ladyzhenskaya and  V. A. Solonnikov,  Determination of solutions of boundary value problems for stationary Stokes and Navier-Stokes equations having an unbounded Dirichlet integral, {\it Zap. Nauchn. Sem. Leningrad. Otdel. Mat. Inst. Steklov. (LOMI)}, {\bf 96} (1980), 117--160.
	
	\bibitem{LRZMA22}Z. Lei, X. Ren, and Q. S. Zhang, A Liouville theorem for axi-symmetric Navier-Stokes equations on $\mathbb{R}^2 \times \mathbb{T}^1$, {\it Math. Ann.}, {\bf383} (2022), 415--431.
	
	\bibitem{LZZSSM17}Z. Lei,  Q. S. Zhang, and N. Zhao, Improved Liouville theorems for axially symmetric Navier-Stokes equations (in Chinese), {\it Sci. Sin. Math.}, {\bf47} (2017), 1183--1198.
	
	\bibitem{LEJMPA33}J. Leray, \'Etude de diverses \'equations int\'egrales non lin\'eaires et de quelques probl\'emes que pose l'hydrodynamique, {\it J. Math. Pures Appl.},
	{\bf 12} (1933),  1--82.
	
	\bibitem{ALPY} Z. Li, X.  H. Pan, and J. Yang,  Characterization of smooth solutions to the Navier-Stokes equations in a pipe with two types of slip boundary conditions,
	{\it arXiv: 2110.02445v3}.
	
	\bibitem{LZCPAA22}C. Li and K. Zhang, A note on the Gagliardo-Nirenberg inequality in bounded domain, {\it Comm. Pure Appl. Anal.},
	{\bf 21} (2022), no. 12,   4013--4017.
	
	\bibitem{MRLARMA12} N. Masmoudi  and  F. Rousset,
	Uniform regularity for the Navier-Stokes equation with Navier boundary condition,
	{\it Arch. Ration. Mech. Anal.}, {\bf203} (2012), no. 2, 529--575.	
	
	
	
	\bibitem{NMARSI27}
	\newblock C. Navier,
	\newblock Sur les lois de l'equilibrie  et du mouvement des corps elastiques,
	{\it Mem. Acad. R. Sci. Inst. France},  {\bf6} (1827) 369.
	
	\bibitem{NazSM90}S. A. Nazarov, Asymptotic solution of the Navier-Stokes problem on the flow of a thin layer of fluid, {\it Siberian Math. J.},  {\bf 31} (1990), 296--307.
	
	\bibitem{SANPJMFM991}S. A. Nazarov and  K. Pileckas, On the solvability of the Stokes and Navier-Stokes problems in the domains that are layer-like at infinity, {\it J. Math. Fluid Mech.},  {\bf 1} (1999), no.  1,  78--116.
	
	\bibitem{SANPJMFM992NE}S. A. Nazarov and  K. Pileckas, The asymptotic properties of the solutions to the Stokes problem in  domains that are layer-like at infinity, {\it J. Math. Fluid Mech.},  {\bf 1} (1999), no.  2,  131--167.
	
	\bibitem{OYASN77}O. A. Oleinik and G. A. Yosifian, Boundary value problems for second order elliptic equations in unbounded domains and Saint-Venant's principle, {\it Ann. Scuola Norm. Sup. Pisa}, Ser. IV, {\bf 4} (1977), no. 2, 269--290.
	
	
	\bibitem{PJMP21}X. Pan,  Liouville theorem of D-solutions to the stationary magnetohydrodynamics system in a slab, {\it J. Math. Phys.}, {\bf 62} (2021), no. 7, Paper No. 071503, 14 pp.
	
	\bibitem{PLNonwor20}X. Pan and Z. Li, Liouville theorem of axially symmetric Navier-Stokes equations with growing velocity at infinity, {\it Nonlinear Anal. Real World Appl.}, {\bf 56} (2020), 103159, 8 pp.
	
	\bibitem{PileJSM84} K. Pileckas,  Existence of solutions for the Navier-Stokes equations having an infinite dissipation of energy in a class of domains with noncompact boundaries, {\it J. Sov. Math.},  {\bf 25} (1984), no. 1, 932--948.
	
	\bibitem{PMSLI2002} K. Pileckas,  Asymptotics of solutions of the stationary Navier-Stokes system of equations in a  domain of layer type, {\it Sbornik: Mathematics},  {\bf 193} (2002), no. 12, 1801--1836.
	
	\bibitem{PMZAAS08} K. Pileckas and M. Specovius--Neugebauer, Artificial boundary conditions for the Stokes and Navier-Stokes equations in domains that are layer-like at infinity, {\it  Z. Anal. Anwend.},  {\bf 27} (2008), no. 2, 125--155.
	
	\bibitem{PMSAA10} K. Pileckas and M. Specovius--Neugebauer, Asymptotics of solutions to the Navier-Stokes system with non-zero flux in a layer-like domain, {\it  Asymptot. Anal.},  {\bf 69} (2010), no. 3--4, 219--231.
	
	\bibitem{PZAAA07} K. Pileckas and L. Zaleskis, Weighted coercive estimates of solutions to the Stokes problem in parabolically growing layer, {\it  Asymptot. Anal.},  {\bf 54} (2007), no. 3--4, 211--233.
	
	
	
	\bibitem{SN16}G. Seregin, Liouville type theorem for stationary Navier-Stokes equations, {\it Nonlinearity}, {\bf 29} (2016), no. 8, 2191--2195.
	
	\bibitem{SWX}K. Sha,  Y. Wang, and C. Xie,
	On the steady Navier-Stokes system with Navier-slip boundary conditions in two-dimensional channels,
	{\it arXiv: 2210.15204v2.}
	
	\bibitem{RATARMA65} R. A. Toupin, Saint-Venant's principle, {\it Arch. Ration. Mech. Anal.}, {\bf 18} (1965), 83--96.
	
	
	\bibitem{TT18}T. P. Tsai, Lectures on Navier-Stokes Equations. Graduate Studies in Mathematics,
	192. American Mathematical Society, Providence, RI, 2018.
	
	\bibitem{WJDE19} W. Wang,
	Remarks on Liouville type theorems for the 3D steady axially symmetric Navier-Stokes equations,
	{\it J. Differ. Equ.}, {\bf266} (2019), 6507--6524.
	
	\bibitem{WXJDE23}
	Y. Wang and C. Xie,
	Uniqueness and uniform structural stability of stability of Poiseuille flows in an infinitely long pipe with Navier boundary conditions,
	{\it J. Differ. Equ.}, {\bf 360} (2023), 1--50.	
	
	
	\bibitem{XXCPAM07} Y. Xiao and Z. Xin,
	On the vanishing viscosity limit for the 3D Navier-Stokes equations with a slip boundary condition,
	{\it Comm. Pure  Appl. Math.}, {\bf60} (2007), no. 7, 1027--1055.
	
	
	\bibitem{ZOP22RE} Q. Zhang and  X. Pan, A Review of results on axially symmetric Navier-Stokes equations, with Addendum by X. Pan and Q. Zhang, {\it Anal. Theory Appl.}, {\bf 38} (2022), no. 3, 243--296.
	
	\bibitem{NZ19} N. Zhao, A Liouville type theorem for axially symmetric D-solutions to steady Navier-Stokes equations, {\it Nonlinear Anal.}, {\bf 187} (2019), 247--258.
	
	
\end{thebibliography}
\end{document}